\documentclass{amsart}
\renewcommand{\emph}[1]{{\it #1}}

\newcommand{\NN}{\mathbb{N}}

\newcommand{\RR}{\mathbb{R}}
\newcommand{\CC}{\mathbb{C}}
\newcommand{\PP}{\mathbb{P}}

\newcommand{\Y}{\mathcal{Y}}

\newcommand{\calM}{{\mathcal M}}

\newcommand{\bs}{\boldsymbol}

\newcommand{\JS}{\mathcal{J}}
\newcommand{\E}{\mathcal{E}_2}
\newcommand{\SE}{\mathcal{S}\hspace{-.4mm}\mathcal{E}_2}
\newcommand{\SA}{\mathcal{S}\hspace{-.4mm}\mathcal{A}_2}
\newcommand{\PGL}{\mathcal{P}\hspace{-.4mm}\mathcal{G}\hspace{-.4mm}\mathcal{L}}
\newcommand{\Sym}{\text{Sym}}
\newcommand{\JE}{\mathcal{J}^{\mathcal{E}}}
\newcommand{\JA}{\mathcal{J}^{\mathcal{S}\hspace{-.4mm}\mathcal{A}}}
\newcommand{\SigE}{\mathcal{S}^{\mathcal{E}}}
\newcommand{\SigA}{\mathcal{S}^{\mathcal{S}\hspace{-.4mm}\mathcal{A}}}

\DeclareMathOperator{\im}{\rm im}

\DeclareMathOperator{\rk}{{\rm rank}}

\DeclareMathOperator{\codim}{{\rm codim}}

\DeclareMathOperator{\TRACK}{\rm TRACK}
\DeclareMathOperator{\sing}{\rm sing}

\newcommand\imclo[1]{\overline{\im #1}}

\newcommand\ratmap{\dashrightarrow}%\, \DashedArrow[->,densely dashed]\,}

\newcommand\diff[2]{(D \, #1)_{#2}}

\newtheorem{problem}{Problem}
\newtheorem{thm}{Theorem}[section]
\newtheorem{prop}[thm]{Proposition}
\newtheorem{proposition}[thm]{Proposition}

\newtheorem{lem}[thm]{Lemma}
\newtheorem{lemma}[thm]{Lemma}
\newtheorem{conj}[thm]{Conjecture}
\theoremstyle{definition}
\newtheorem{example}[thm]{Example}
\newtheorem{defn}[thm]{Definition}
\newtheorem{definition}[thm]{Definition}
\newtheorem{remark}[thm]{Remark}
\newtheorem{algorithm}{Algorithm}

\usepackage{amsmath,amssymb,amsfonts,amsthm,algorithmic,mathtools,thmtools,faktor,tikz,array,enumitem,hyperref,float}

\AtBeginDocument{%
  \providecommand\BibTeX{{%
    \normalfont B\kern-0.5em{\scshape i\kern-0.25em b}\kern-0.8em\TeX}}}

\calclayout

\begin{document}
\title{Signatures of algebraic curves\\ via numerical algebraic geometry}
\author{Timothy Duff}
\address{Timothy Duff\newline \indent School of Mathematics, Georgia Tech}
\email{tduff3@gatech.edu}
%\affiliation{
%  \institution{School of Mathematics, Georgia Tech}
 % \streetaddress{686 Cherry St.~NW}
  %\city{Atlanta}
  %\state{Georgia, USA}
  %\postcode{30308}

\author{Michael Ruddy}
\address{Michael Ruddy\newline \indent Data Institute at University of San Francisco}
\email{mruddy@usfca.edu}

\renewcommand{\algorithmicrequire}{\textbf{Input:}}
\renewcommand{\algorithmicensure}{\textbf{Output:}}

\newcommand\tim[1]{\textcolor{red}{T: #1}}
\newcommand\mike[1]{\textcolor{blue}{M: #1}}

\begin{abstract}
We apply numerical algebraic geometry to the invariant-theoretic problem of detecting symmetries between two plane algebraic curves. We describe an efficient equality test which determines, with ``probability-one'', whether or not two rational maps have the same image up to Zariski closure. The application to invariant theory is based on the construction of suitable signature maps associated to a group acting linearly on the respective curves. We consider two versions of this construction: differential and joint signature maps. In our examples and computational experiments, we focus on the complex Euclidean group, and introduce an algebraic joint signature that we prove determines equivalence of curves under this action and the size of a curve's symmetry group. We demonstrate that the test is efficient and use it to empirically compare the sensitivity of differential and joint signatures to different types of noise.
\end{abstract}

%  \label{fig:teaser}
%
%\end{teaserfigure}

%%
%% This command processes the author and affiliation and title
%% information and builds the first part of the formatted document.
\maketitle
\section{Introduction}\label{sec:intro}

The study of plane curves under linear group actions is a classical subject of both differential \cite{Gugg} and algebraic geometry \cite{O99} with applications to image science \cite{Z94}. In particular, an important problem is to determine whether two curves are equivalent under such a group action, which is more difficult when there is a significant level of noise. For instance, when the transformation group is the group of rigid motions, this can translate to deciding whether two contours represent the same object in different positions, or, in the case of affine and projective transformations, whether two contours might correspond to different projections of the same 3D object. For plane algebraic curves, we state the \emph{group equivalence problem} as follows:

\begin{problem}
\label{problem:group}
Given a positive dimensional algebraic group $G\subset \PGL_3(\CC)$ acting linearly on $\mathbb{C}^2$ and two plane algebraic curves $C_0, C_1\subset \CC^2$, decide if there exists $g\in G$ such that $C_0=\overline{g\cdot C_1}$.
\end{problem}

There exist many different symbolic algorithms to determine equivalence under a particular group of algebraic transformations. For instance, one can construct a set rational invariants that the characterize the orbits of the action on the coefficients of curves of fixed degree $d$ \cite{DK15, HK07, S08} or a pair of rational differential invariants which define a signature polynomial on a curve characterizing its equivalence class \cite{BKH13, KRV18}. However these approaches usually rely on Gr\"obner basis computations which can become increasingly difficult as the degree of the curve increases.

In the analagous setting of \emph{smooth} curves in $\RR^2,$ the Fels-Olver moving frame method~\cite{FO99}, based on Cartan's method of moving frames, associates to each curve a \emph{differential signature curve}, defined in terms of smooth invariants, which is classifying for the group action.
In greater generality, differential signatures may be constructed for smooth submanifolds of some ambient space equipped with a Lie group action.
The differential signature locally characterizes the manifold's equivalence class under the action, meaning that manifolds with the same signature are locally equivalent under the Lie group~\cite{FO99}.

Differential signatures of curves have been successfully applied to object recognition under noise, with applications ranging from jigsaw puzzle reconstruction~\cite{HO14} to medical imaging~\cite{GS17}. Differential signatures have also been used to solve classical invariant theory problems such as determining equivalence of binary and ternary forms~\cite{BO00,KM02,O99}. In~\cite{BKH13} the notion of a signature polynomial was introduced to determine equivalence of plane algebraic curves, and in~\cite{KRV18} it is shown that this reduction to Problem~\ref{problem:map} can always be done.

In the algebraic setting, the differential signature construction was adapted for constructing algebraic invariants in~\cite{HK07} and rational invariants in~\cite{KRV18}. 
For an algebraic group acting on $\CC^2$ and a plane curve $C\subset \CC^2,$ the signature curve is the image of a rational map $\Phi :\CC^2 \ratmap \CC^2.$

\begin{figure}
  \label{fig:teaser}
\begin{center}
\begin{tabular}{m{11em} m{1em} m{11em} m{2em} m{5em}}
  \includegraphics[width=0.2\textwidth]{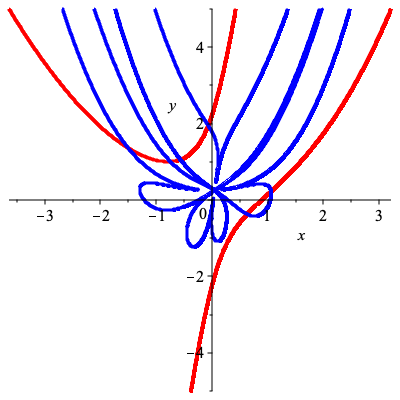}
  &
 $\cong $
&
  \includegraphics[width=0.2\textwidth]{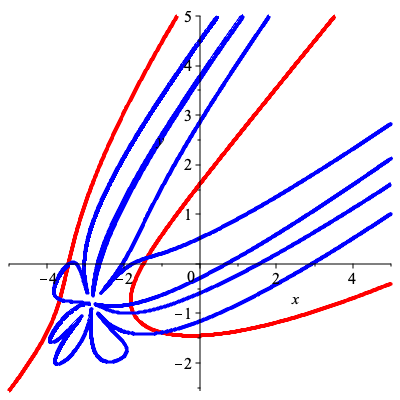}
  &
     $\xrightarrow{\Phi} $
   &
   \includegraphics[width=0.2\textwidth]{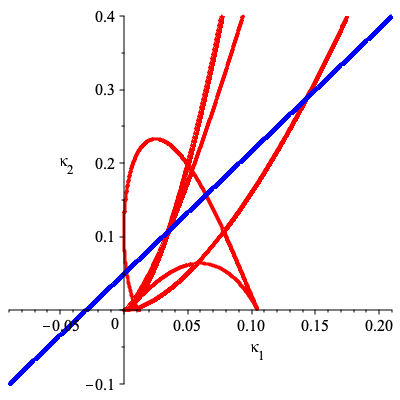}
 \end{tabular}

\end{center}
\caption{Two curves and their signature in red. A line and its pullback in blue.}
\end{figure}
\begin{example}
In Figure~\ref{fig:teaser}, the red curve on left depicts real points $(x,y)$ such that $8 x^3 - 20  x  y + 2y^2 + 5x - 10=0.$ Applying a real rotation and translation yields the curve in the middle. Thus these curves are equivalent under the linear action of the complex Euclidean group $\mathcal{E}_2(\CC).$ The closed image of their respective differential signature maps is the red curve of degree $48$ depicted on the right.
\end{example}

In the setting of Problem~\ref{problem:group}, \cite{KRV18} observed that local equivalence implies global equivalence, reducing Problem~\ref{problem:group} to a special case of Problem~\ref{problem:map} below.

\begin{problem}
\label{problem:map}
Given two irreducible algebraic varieties, $X_0\subset \CC^{n_0}$ and $X_1 \subset  \CC^{n_1},$ and rational maps, $\Phi_0 : X_0 \ratmap \CC^m$ and $\Phi_1 : X_1 \ratmap \CC^m,$ decide if $\imclo{\Phi_0} = \imclo{\Phi_1}.$\footnote{In Problem~\ref{problem:map}, $\imclo{\Phi_i }$ denotes the Zariski closure of the image of $\Phi_i.$ We do not address the more delicate problem of deciding equality of the constructible sets $\im \Phi_i.$}
\end{problem}

In the smooth setting, the reduction of Problem~\ref{problem:group} to Problem~\ref{problem:map} can also be achieved through the use of \textit{joint signatures} (introduced in \cite{O01}) which are obtained by constructing maps using joint invariants of the induced action of $G$ on the product $\mathbb{C}^2\times\hdots\times\mathbb{C}^2$.
The joint signatures may be interpreted as $0$-th order differential invariants, and are considered to be more noise-resistant in applications.
This further motivates our interest in studying Problem~\ref{problem:map} in full generality.

Our approach to Problem~\ref{problem:map} is via \emph{numerical algebraic geometry}.
It is much in the same spirit as previous works~\cite{CK19,HS10,HS13},
where the cost of implicitization is replaced by the cost of computing certain \emph{witness sets} (more precisely, pseudowitness sets.) 
This allows us to study differential and joint signatures for curves up to a much higher degree than in previous works.
To accommodate both differential and joint signatures, we give our main algorithm (Algorithm~\ref{alg:equality-test}) as a general solution to Problem~\ref{problem:map}. 
Algorithm~\ref{alg:equality-test} is a variant of the classical homotopy membership test, specialized to Problem~\ref{problem:map}.
As another novel aspect, we consider the use of recently introduced \emph{multiprojective witness sets}~\cite{HLRS} in Algorithm~\ref{alg:equality-test}.

The authors of this work submitted a preliminary report of this work in the conference proceedings of ISSAC 2020 \cite{DR20}.
In this version we added significantly more details as well as examples, with the goal of writing a paper more accessible than the conference version to researchers interested in either signatures or numerical algebraic geometry.
We also add a rigorous characterization of a planar algebraic curve's symmetry group under the Euclidean group $\mathcal{E}_2(\CC)$ via the joint signature map, mirroring previous results \cite{KRV18} for the differential signature map.
We conduct several new experiments on the sensitivity of the numerical equality test to noise, now involving the equi-affine group as well as curves computed from noisey samples.
Finally, we include a discussion of the relationship between monodromy and the symmetry groups of curves.

The paper is organized as follows. In Section~\ref{sec:sig} we discuss signatures of planar curves and how they can be used to reduce Problem~\ref{problem:map} to Problem~\ref{problem:map}. In~\ref{sec:diffsig} we follow the construction in~\cite{BKH13, KRV18} to describe a differential signature for plane algebraic curves using a \textit{classifying pair} of differential invariants. In~\ref{sec:jointsig} we describe how joint signatures can be used to determine equivalence of plane curves using lower order differential invariant functions, with a detailed analysis in the case of the complex Euclidean group $\mathcal{E}_2(\CC)$. In Section~\ref{sec:nag}, we review notions from numerical algebraic geometry and describe a general solution to Problem~\ref{problem:map} (Algorithm~\ref{alg:equality-test}). In Section~\ref{sec:implementation}, we describe an implementation in Macaulay2~\cite{M2}, which has been successful for studying both classes of maps on curves of degree up to $10.$ Our (reproducible) experiments show that offline witness computation for plane curves of various degrees is feasible, that the online equality test gives a fast alternative to symbolic methods, and that the numerical approach is robust in a certain regime of noise. Additionally we investigate how different types of noise affect the sensitivity of the numerical equality test. %\footnote{Obtain the code at \url{https://github.com/timduff35/NumericalSignatures}.}

\section{Signatures of curves}\label{sec:sig}

\subsection{Invariants of planar curves}\label{sec:diffex}

A classical subject in differential geometry are invariants and the classification of differentiable planar curves in $\RR^2$ under rigid motions \cite{Gugg}. This can be seen as a variant on what we defined as Problem \ref{problem:group}.

\begin{defn}\label{def:Gequiv}
Two curves $C_0, C_1$ are said to be \textit{$G$-equivalent}, denoted $C_0\cong_G C_1,$ if there exists an element $g$ in the group of transformations $G$ such that $C_0=g\cdot C_1$.
\end{defn}

Remaining purposefully agnostic about what constitutes a ``curve'' and a ``group of transformations,'' we can define the \textit{Group equivalence problem for curves} as: given two curves and a group of transformations $G$ decide if they are $G$-equivalent. In this context both classical questions about geometry of real curves and Problem \ref{problem:group} are both specific instances of a larger class of problems. In this subsection, we discuss how previous work on the group equivalence problem for differentiable curves in $\RR^2$ connects to our approach to Problem \ref{problem:group} for algebraic curves in $\CC^2$.

For now let $C$ refer to the image of a smooth\footnote{Here smooth refers to a map defined by infinitely differentiable functions. For simplicity we require smooth functions, though for the results and constructions referenced, this restriction can be loosened to $n$-differentiable for an appropriate choice of $n$.} 
 map $\gamma = (x(t), y(t))$ where $\gamma : I \rightarrow \RR^2$ for some interval $I\subset \RR$. We denote $\SE(\RR)$ as the \textit{special Euclidean group}, the transformation group of rotations and translations of $\RR^2$. A classical invariant\footnote{More rigorously, Euclidean curvature is a \textit{local} invariant, and one must take into account sign changes, i.e. a rotation of $\pi$ changes the sign of $\kappa(t)$.}
 of curves under rigid motion is the Euclidean curvature function $\kappa(t)$, defined below in \eqref{eq:Ecurvature}, meaning that the value of curvature at a particular point of a curve does not change when the curve transformed by $\SE(\RR)$.
 
\begin{equation}\label{eq:Ecurvature}
\kappa(t) = \frac{x'(t)y''(t) - y'(t)x''(t)}{(x'(t)^2+y'(t)^2)^{3/2}}
\end{equation}

Euclidean curvature at a point on a curve can be defined in many ``geometrically-satisfying '' ways, as the multiplicative inverse of the radius of the osculating circle, or norm of the tangent vector when the curve is parameterized by arc length. Euclidean curvature also provides a way to solve the group equivalence problem for curves under $\SE(\RR)$. The following theorem appears in many places, for instance \cite{Gugg}, and is sometimes referred to as the ``Fundamental theorem for planar curves.''

\begin{thm}\label{thm:fundmentalEcurvature}
If two smooth curves have the same Euclidean curvature as a function of arc length, then they are $\SE(\RR)$-equivalent.
\end{thm}

Thus $\kappa(s)$, Euclidean curvature when $C$ is parameterized by the arc length parameter $s$, completely determines a curve up to $\SE(\RR)$. In practice comparing curves' curvature functions to determine $\SE(\RR)$-equivalence is difficult as this comparison depends on the parameterization and the starting point when the curve is closed. In this sense, a single curve can have infinitely many different curvature functions. Motivated by applications to object recognition, the authors of \cite{CO98} proposed the use of the \textit{Euclidean signature curve} to determine $\SE(\RR)$-equivalence of smooth curves.

\begin{definition}\label{def:EuclideanSigSmoth}
The \textit{Euclidean signature curve} of a smooth curve $C$ is the image of $C$ under the map $\mathcal{S} : C \rightarrow \RR^2$ defined by $\mathcal{S} = (\kappa, \kappa_s)$, where $\kappa_s$ is the function representing the derivative of $\kappa$ with respect to arc length.
\end{definition}

\begin{thm}[Theorem 2.3 in \cite{CO98}]
If two smooth curves, $C_0, C_1$ have the same Euclidean signature curve, then they are locally equivalent under $\SE(\RR)$.
\end{thm}

Here locally equivalent means that around each point of $C_0$, there exists open subsets $U_0 \subset C_0$ and $U_1 \subset C_1$ such that $U_0  = g\cdot U_1$ for some $g\in \SE(\RR)$. Thus the local geometry of a curve is determined by relationship between Euclidean curvature and another differential invariant function, the derivative of $\kappa$ with respect to arc length. For closed curves this comparison is parameterization and starting point invariant. In subsequent works, the Euclidean signature curve is used for a curve matching algorithm \cite{HO13} and applied to automatic jigsaw puzzle reassembly \cite{HO14}.

The authors of \cite{CO98} also note that this procedure generalizes to planar curves under other transformation groups of $\RR^2$. For most Lie group actions of $G$ on $\RR^2$ there exists a notion of $G$-invariant curvature $\kappa$ and $G$-invariant arc length $s$ such that if curves have the same image, or \textit{differential signature}, under $(\kappa, \kappa_s)$ then they are locally equivalent under $G$ \cite[Thm 5.2]{CO98}. Moreover a pair of such differential invariants can be constructed explicitly by the Fels-Olver moving frame method \cite{FO99}, giving a practical method to locally solve the group equivalence problem for smooth curves.

Turning attention back to Problem \ref{problem:group}, for algebraic curves local equivalence under a group $G$ immediately implies global $G$-equivalence as in Definition \ref{def:Gequiv}. Thus the differential signature characterizes algebraic curves under a transformation group $G$. In \cite{BKH13} the authors connect the differential signature to symbolic methods by noticing that when the differential invariants can be expressed as a rational map on the curve, two algebraic curves' differential signatures can be compared by computing their implicit equations, connecting Problem \ref{problem:map} to the group equivalence problem for algebraic curves over $\RR^2$.

In \cite{KRV18} it is shown that for any subgroup of $\PGL_3(\CC)$ there exists a pair of rational differential invariants which can reduce Problem \ref{problem:group} to Problem \ref{problem:map}. Thus the differential signature can be used to solve questions of classical invariant theory in a way that uses the same invariants regardless of the degree of the algebraic curves in question. Moreover these differential invariants can be interpreted as generators of a field of rational invariants, meaning that they can be computed by symbolic methods such as those in \cite{DK15} or the cross-section method in \cite{HK07} inspired by the previously mentioned moving frame method. Thus Problem \ref{problem:group} can be solved end-to-end by symbolic computation in this way.

In practice, these methods can be quite slow for computing implicit equations of differential signatures (see \cite[Ex 3.2.13]{RuddyThesis} for instance), which is our motivation to further extend the connection between Problem \ref{problem:group} to Problem \ref{problem:map} by leveraging numerical algorithms for computing pseudo-witness sets to compare differential signatures of algebraic curves. In Section \ref{sec:diffsig} we explain, in greater detail, the reduction of Problem \ref{problem:group} to Problem \ref{problem:map} along with examples.

In addition to differential invariants, a similar approach has been taken using \textit{joint} differential invariants to solve the group equivalence problem for smooth curves in \cite{O01}, which are in theory more robust to noise and perturbations. In Section \ref{sec:jointsig} we consider for the first time using the joint signature in a completely algebraic approach to Problem \ref{problem:group}, proving in the case of $\E(\CC)$ that the Euclidean joint signature characterizes equivalence classes of algebraic curves.

In the next two sections, we assume that plane curves are complex algebraic, irreducible, and of degree greater than one. The degree restriction removes from consideration lines, on which not all transformations $g\in G$ may not be defined.

\subsection{Differential signatures}\label{sec:diffsig}

For algebraic curves, we tweak Definition \ref{def:Gequiv} to the following definition, which allows for the image of a curve under the action of $G$ to not be closed.

\begin{defn}\label{def:GequivALG}
Two algebraic curves $C_1$ and $C_2$ are \textit{$G$-equivalent} if there exists $g\in G$ such that $C_1 = \overline{g\cdot C_2}$.
\end{defn}

We assume that the group $G\subset \PGL_3(\CC)$ is a positive dimensional algebraic group acting linearly on $\mathbb{C}^2$ with action $g\cdot (x,y)=(\overline{x},\overline{y})$.

\begin{defn}
The \textit{projective group} $\PGL_3(\CC)$ is the group of invertible matrices modulo scaling, i.e. $\PGL_3(\CC) \cong \faktor{\mathcal{G}\mathcal{L}_3(\CC)}{\lambda I}$. The linear action on $\CC^2$ is defined by the map $\Phi : \PGL_3(\CC) \times \CC^2 \dashrightarrow \CC^2$ where for $A\in \PGL_3(\CC)$ and $p=(x,y)\in \CC^2$,
$$
\Phi(A,p) = \left(\frac{a_{11}x+a_{12}y+a_{13}}{a_{31}x+a_{32}y+a_{33}},\frac{a_{21}x+a_{22}y+a_{23}}{a_{31}x+a_{32}y+a_{33}}\right).
$$
\end{defn}

We consider a few classical subgroups of $\PGL_3(\CC)$.

\begin{defn}\label{def:EuclideanGroup}
The \textit{Euclidean group} $\E(\CC)$ is the subgroup of $\PGL_3(\CC)$ given by matrices of the form
$$
\begin{bmatrix}
\alpha c & \alpha s & a \\
-s & c & b \\
0 & 0 & 1
\end{bmatrix}
$$
where $a,b,c,s\in \CC$, $\alpha \in  \{-1, 1\}$, and $c^2+s^2=1$.
\end{defn}

\begin{defn}
The \textit{special Euclidean group} $\SE(\CC)$ is the subgroup of $\E(\CC)$ consisting of determinant one matrices.
\end{defn}

\begin{defn}\label{def:EquiAffineGroup}
The \textit{equi-affine group} $\SA(\CC)$ is the subgroup of $\PGL_3(\CC)$ given by matrices of the form
$$
\begin{bmatrix}
a_{11} & a_{12} & a_{13} \\
a_{21} & a_{22} & a_{23} \\
0 & 0 & 1
\end{bmatrix}
$$
with entries in $\CC$ and $a_{11}a_{22}-a_{12}a_{21}=1$.
\end{defn}

\noindent A differential signature that determines $G$-equivalence of algebraic curves can be constructed from a set of classifying invariants (Definition~\ref{def:classInv}). We let $J^n$ denote the $n$th order jet space, a complex vector space of dimension $(n+2)$ with coordinates $(x,y,y^{(1)},\hdots, y^{(n)}).$ Letting $\Omega (J^n)$ denote the set of complex-differentiable functions from $J^n$ to $\CC,$ the \emph{total derivative operator} $\frac{d}{dx} : \Omega (J^n) \to \Omega (J^{n+1})$ is the unique $\CC$-linear map satisfying the product rule and the relations $\frac{d}{dx}(x)=1$, $\frac{d}{dx} (y^{(k)})=y^{(k+1)}$ for $k\geq 0$, cf.~\cite[Ch.~7]{Olv95}. The \textit{prolonged} action of $G$ on $J^n$ is given by 
$$
g\cdot (x,y,y^{(1)},\hdots , y^{(n)})=(\overline{x},\overline{y},\overline{y}^{(1)},\hdots , \overline{y}^{(n)})
$$
where
\begin{equation*}
\overline{y}^{(1)}= \frac{\frac{d}{dx}\left[ \overline{y}(g,x,y)\right]}{\frac{d}{dx}\left[ \overline{x}(g,x,y)\right]},\quad 
 \overline{y}^{(k+1)}= \frac{\frac{d}{dx}\left[ \overline{y}^{(k)}(g,x,y,y^{(1)},\hdots,{ y^{(k)}})\right]}{\frac{d}{dx}\left[ \overline{x}(g,x,y)\right]}
 \text{  for } k=1,\hdots,n-1.
\end{equation*}

\begin{defn}
\label{def:diffInv}
A \textit{differential invariant} for the action of $G$ is a function on $J^n$ that is invariant under the prolonged action of $G$ on $J^n$. The \textit{order} of a differential invariant is the maximum $k$ such that the function depends explicitly on $y^{(k)}$.
\end{defn}

\begin{defn}
The \textit{$n$-th jet} of an algebraic curve $C$ is the image of the map $j^n_C:C\ratmap J^n$ given (where defined) by
$$
(x,y) \mapsto (x,y,y_C^{(1)}(x,y),y_C^{(2)}(x,y),\hdots, y_C^{(n)}(x,y)),
$$
where $y^{(k)}_C(x,y)$ is the $k$-th derivative of $y$ with respect to $x$ at the point $(x,y)\in C$.

\end{defn}
\noindent The prolonged action of $G$ is defined such that
$$
g\cdot j^n_C(C) = j^n_{g\cdot C}(g\cdot C).
$$
\begin{defn}
The \textit{restriction} of a differential invariant $K$ of order $n$ to a curve $C$ is the map $K|_C:C\ratmap \mathbb{C}^2$ given by $K|_C=K\circ j^n_C$.
\end{defn}
\noindent The coordinates of the $n$-th jet map $j^n_C$ are rational functions of $x$ and $y$ that can be computed via implicit differentiation: 
\begin{equation}
\label{eq:SLP}
y^{(1)}_C=\frac{-\partial_x \, F}{\partial_y \, F}\quad\text{and}\quad y^{(k+1)}_C = \partial_x \, y_C^{(k)} +\partial_y \, y_C^{(k)} \,  y_C^{(1)}.
\end{equation}
where $\mathcal{I}_C = \langle F \rangle .$ Thus, if $K$ is a \textit{rational} differential invariant of order $n$, meaning it is a rational function in the coordinates of $J^n$, then $K|_C$ is a rational function in $x$ and $y$.

\begin{defn}
We say that a set of differential invariants $\mathcal{I}$ \textit{separates orbits} for the prolonged action on a nonempty Zariski-open $W\subset J^n$ if, for all $p,q\in W$,
$$
K(p)=K(q)\,\,\, \forall K\in \mathcal{I} \quad \Leftrightarrow \quad \exists g\in G\,\, \text{such that} \,\, p=g\cdot q.
$$
\end{defn}

\begin{example}\label{ex:ProlongedEuclidean}
The prolonged action of $\E(\CC)$ on $J^2$ is given by

$$
g\cdot (x,y,y_x, y_{xx}) = \left(\alpha (c x+sy)+a, -sx+cy+b, \alpha \frac{cy_x-s}{sy_x+c},\frac{\alpha y_{xx}}{(sy_x+c)^3}\right),
$$

\noindent where $y_x = y^{(1)}$ and $y_{xx} = y^{(2)}$. The Euclidean curvature function in \eqref{eq:Ecurvature} can be written in the coordinates of $J^2$ as

$$
\kappa(x,y,y_x y_{xx}) = \frac{y_{xx}}{\left(1+y_x^2\right)^{3/2}}.
$$

Though $\kappa$ is not a rational differential function on $J^2$, the function $\kappa^2$ is. Thus $\kappa^2$ is a rational differential invariant function for the action of $\E(\CC)$. In fact one can show that $\kappa^2$ separates orbits for the prolonged action of $\E(\CC)$ on $J^2$. For a particular algebraic curve $C$ defined by $F(x,y)=0$, we can restrict $\kappa^2$ to $C$ to obtain the map $\kappa^2|_C:C\rightarrow \CC$ defined by

$$
\frac{\left(\frac{-F_{xx}F_y^2+2F_{xy}F_xF_y-F_{yy}F_x^2}{F_y^3}\right)^2}{\left(1+\left(\frac{F_x}{F_y}\right)^2\right)^3} = \frac{\left(-F_{xx}F_y^2+2F_{xy}F_xF_y-F_{yy}F_x^2\right)^2}{\left(F_x^2+F_y^2\right)^3}.
$$

\end{example}

\begin{defn}
\label{def:classInv}
Let an $r$-dimensional algebraic group $G$ act on $\mathbb{C}^2$. A pair of rational differential invariants $\mathcal{I}=\{K_1,K_2\}$ is said to be \textit{classifying} if $K_1$ separates orbits on $U_k \subset J^k$ for some $k<r$ and $\mathcal{I}$ separates orbits on $U_r \subset J^r$.
\end{defn}

\noindent For a particular action of $G$, such a pair of classifying invariants always exists, and one can explicitly construct a pair by computing generators for the field of rational invariants for the prolonged action of $G$ \cite[Thm 2.20]{KRV18}, using algorithms such as those found in \cite{DK15} and \cite{HK07}. It should be noted that $\mathcal{I}$ is not unique, and different choices can lead to different differential signatures.

\begin{defn}
For a pair of classifying invariants $\mathcal{I}=\{K_1,K_2\}$, an algebraic curve $C$ is said to be \textit{non-exceptional} if all but finitely many points on $p\in C$ satisfy
\begin{displaymath}
j^k_C(p)\in U_k,\, \,   j^r_C(p)\in U_r, \,  \text{ and } \, 
%\end{displaymath}
%\begin{displaymath}
\frac{\partial K_1}{\partial y^{(k)}}, \, \frac{\partial K_2}{\partial y^{(r)}}\neq 0 \text{ at } j^r_C(p).
\end{displaymath}
\end{defn}
\noindent A generic curve of degree $d$ where $\binom{d+2}{2}-2\geq r$ is non-exceptional with respect to a given classifying set \cite[Thm 2.27]{KRV18}.
\begin{defn}
Let $\mathcal{I}=\{K_1,K_2\}$ be a pair of classifying invariants for the action of $G$ on $\mathbb{C}^2$ and $C$ a non-exceptional algebraic curve with respect to $\mathcal{I}$. The map  $\sigma_C: C \rightarrow \CC^2$ defined by $\sigma_C =\left(K_1|_C,K_2|_C\right)$ is the \textit{differential signature map} for $C$ and its image is the \textit{differential signature} of $C$, denoted $\mathcal{S}_C$.
\end{defn}

\noindent The following appears as Theorem 2.37 in \cite{KRV18}.

\begin{thm}\label{Thm:Signature}
If algebraic curves $C_0 ,C_1$ are non-exceptional with respect to a classifying set of rational differential invariants $\mathcal{I}=\{K_1,K_2\}$ under an action of $G$ on $\mathbb{C}^2$ then
$$
C_0\cong_G C_1 \quad \Leftrightarrow \quad \overline{\mathcal{S}_{C_0}}=\overline{\mathcal{S}_{C_1}}.
$$
\end{thm}

\noindent Since the Zariski-closure of the differential signature of an algebraic curve $\overline{\mathcal{S}_C}$ characterizes its equivalence class under $G$, so does the polynomial vanishing on $\mathcal{S}_C$, referred to as the \textit{signature polynomial} of $C$ and denoted $S_C(\kappa_1,\kappa_2)$. Thus to determine if curves $C_1$ and $C_2$ are $G$-equivalent we can compare signature polynomials $S_{C_1}$ and $S_{C_2}$. The differential signature map also characterizes the size of the symmetry group of $C$ under $G$.

\begin{defn}
The \textit{symemtry group} of $C$ under $G$ is the subgroup of $G$ defined by

$$
\text{Sym}(C,G) = \{g\in G\, |\, C = \overline{g\cdot C}\}.
$$
\end{defn}

\noindent The following follows from Lemma 2.34 and Theorem 2.38  in \cite{KRV18}.

\begin{thm}\label{thm:diffsigsymgroup}
For an algebraic curve, non-exceptional with respect to $\mathcal{I} = \{K_1,K_2\}$, the symmetry group $\Sym(C,G)$ is of cardinality $n < \infty$ if and only if the map $\sigma_C$ is generically $n:1$. Furthermore $\Sym(C,G)$ is infinite if and only if $S_C$ is a single point.
\end{thm}

\begin{example}\label{Ex:E2Inv}

Consider the action of the Euclidean group $\E(\CC)$ on curves in $\CC^2$ (defined in Definition \ref{def:EuclideanGroup}). In Example \ref{ex:ProlongedEuclidean} we saw that $\kappa^2$ is rational invariant for the prolonged action of $\E(\CC)$ on $J^2$. Similarly the function
$$
\kappa_s^2 = \frac{\left(y_{xxx}\left(1+y_x^2\right)-3y_xy_{xx}^2\right)^2}{\left(1+y_x^2\right)^6},
$$
\noindent representing the square of the derivative of curvature with respect to arc length, is also a rational invariant for this action. Together the pair $\mathcal{I} = \{ \kappa^2, \kappa_s^2\}$ is a classifying set of rational differential invariants for the action of $\E(\CC)$ on curves in $\CC^2$. Moreover, there are no $\mathcal{I}$-exceptional algebraic curves---for details see \cite[Sec 4.1]{RuddyThesis}. By Theorem~\ref{Thm:Signature}, the equivalence class of an algebraic curve $C$ under $\mathcal{E}_2(\CC)$ is completely determined by $\overline{\mathcal{S}_C}$.

Consider the two ellipses $C_1$ and $C_2$ defined by the zero sets of
\begin{align*}
F_1(x,y)&=x^2 + y^2 +xy-1\\
F_2(x,y)&=x^2+y^2 - xy -5x +y +6,
\end{align*}
\noindent respectively. The signature maps $\sigma_{C_1}$, $\sigma_{C_2}$ are rational maps on $C_1, C_2$ defined by
\begin{center}
\begin{tabular}{ccl}
$\sigma_{C_1}(x,y)$ 
$= \Big(\kappa^2|_{C_1}, \kappa^2_s|_{C_1}\Big)$& $=$ &
$\Big(\frac{36 \, (F_1(x,y)+1)^2}{(5x^2 + 8 x y + 5y^2)^3},$ 
$ \frac{2916 \, (x-y)^2 \, (x+y)^2 \, (F_1(x,y)+1)^2}{(5x^2 + 8 xy + 5y^2)^6}\Big)$,\\
$\sigma_{C_2}(x,y)$  
$= \Big(\kappa^2|_{C_2}, \kappa^2_s|_{C_2}\Big) $ & $=$ &
$\Big(\frac{36 \, (F_2(x,y)+1)^2}{(5x - 8 xy + 5y^2 - 22 x + 14 y + 26)^3},$
$  \frac{2916 \, (x-y-2)^2 \, (x+y-4)^2 \, (F_2(x,y)+1)^2}{(5x - 8 xy + 5y^2 - 22 x + 14 y + 26)^6}\Big).$
\end{tabular}
\end{center}

\noindent From the above, we see that each map $\sigma_{C_i}$ has an equivalent expression moduolo $\mathcal{I}_{C_i}$ where the total degrees drop by $4.$\footnote{We detect this automatically with an implementation of rational function simplification from~\cite{ratsimp} in Macaulay2~\cite{M2}.} Both ellipses $C_1$ and $C_2$ have symmetry groups under $\E(\CC)$ of cardinality 4 (generated by a reflection and $180^{\circ}$-degree rotation). Thus, by Theorem \ref{thm:diffsigsymgroup}, the above maps are generically $4:1$. We can directly compute the signature polynomials $S_{C_1}$ and $S_{C_2}$ using elimination:
\begin{align*}
S_{C_1}(\kappa_1,\kappa_2) = S_{C_2}(\kappa_2,\kappa_2) = 2916\,{k}_{1}^{6}-13608\,{k}_{1}^{5}+972\,{k}_{1}^{4}{k}_{2}+2187\,{k}_{
      1}^{4}+1944\,{k}_{1}^{3}{k}_{2}+108\,{k}_{1}^{2}{k}_{2}^{2}+4\,{k}_{2}^{3
      }.
\end{align*}

Since these two curves have the same signature polynomial, the Zariski-closure of their images are equal, i.e. $\overline{\mathcal{S}_{C_1}}=\overline{\mathcal{S}_{C_2}}.$ Thus by Theorem \ref{Thm:Signature} the two curves are $\E(\CC)$-equivalent.

\end{example}

\begin{example}\label{ex:EqAffDiff}
For the action of the equi-affine group $\SA(\CC)$ on curves in $\CC^2$ (defined in Definition \ref{def:EquiAffineGroup}), we can again construct a differential signature map from rational differential invariants. The following pair
$$
(K_1,K_2) = \left(\frac{\left(3y^{(4)}y^{(2)}-5\left(y^{(3)}\right)^2\right)^3}{\left(y^{(2)}\right)^8},\frac{9y^{(5)}\left(y^{(2)}\right)^2-45y^{(4)}y^{(3)}y^{(2)}+40\left(y^{(3)}\right)^3}{\left(y^{(2)}\right)^4}\right)
$$
forms a classifying set of rational differential invariants. Here $K_1 = \mu^3$ where $\mu$ is classical affine curvature \cite{Gugg}. For details on classifying sets of rational differential invariants for $\SA(\CC)$ and other classical linear groups see \cite[Sec. 4.1]{RuddyThesis}
\end{example}

\subsection{Joint signatures}\label{sec:jointsig}
In \cite{O01}, the author considers the use of \textit{joint} differential signatures to determine equivalence. For the action of $G$ on $\mathbb{C}^2$ given by $g\cdot (x,y)=(\overline{x},\overline{y})$, consider the induced action on the Cartesian product space $(\mathbb{C}^2)^{n}=\mathbb{C}^2\times \mathbb{C}^2\times \hdots\times\mathbb{C}^2$ given by
$$
g\cdot (x_1,y_1,x_2,y_2,\hdots, x_n,y_n)=(\overline{x}_1,\overline{y}_1,\overline{x}_2,\overline{y}_2,\hdots, \overline{x}_n,\overline{y}_n)
$$
where $\overline{x}_i=\overline{x}|_{x=x_i,y=y_i}$ and $\overline{y}_i=\overline{y}|_{x=x_i,y=y_i}$. For a curve $C\subset\mathbb{C}^2$ denote the Cartesian product by $C^{ n}=C\times C\times\hdots\times C\subset (\mathbb{C}^2)^{n}$. Then we can see that two curves $C_0$ and $C_1$ are $G$-equivalent if and only if their Cartesian products $C_0^{ n},C_1^{ n}$ are $G$-equivalent under the induced action on $(\mathbb{C}^2)^{ n}$.

The advantage of considering $G$-equivalence of products of the curve $C$ is that the order of the differential invariants needed to define a differential signature on this space can be reduced. Though the number of invariants required may increase, the lower order of the differential invariants may result in a more noise-resistant differential signature. In fact, for a large enough product space, it is often possible to construct a differential signature from `$0$-th order' differential invariants, or \textit{joint invariants}, which we refer to as a \textit{joint signature}.

Consider the action of $\mathcal{E}_2(\CC)$ on $\mathbb{C}^2$ as defined in Definition \ref{def:EuclideanGroup}. This induces a diagonal action on the product space $(\mathbb{C}^2)^{n}$ whose joint invariants are the squared inter-point distance functions
$$
d_{jk}(x_j,y_j,x_k,y_k)=(x_j-x_k)^2+(y_j-y_k)^2,
$$
where $j< k$ and $j,k\in \{1,\hdots,n\}$. Let the map $d_n: C^n\rightarrow \CC^{n(n-1)/2}$ be the map which takes an $n$-tuple of points on $C$ and outputs all the inter-point distances, i.e.
\begin{equation}\label{eq:dmap}
(x_1,y_1,\hdots, x_n,y_n)\mapsto (d_{12},d_{13},\hdots, d_{1n},\hdots, d_{(n-1)n}).
\end{equation}
Additionally let $W_n$ be the Zariski-open subset of $ (\mathbb{C}^2)^n$ where all the inter-point distances do not vanish:
$$
W_n = \{ p\in (\mathbb{C}^2)^n\,|\, d_{jk}(p)\neq 0\,\,\, \text{for $j<k$ and $j,k\in \{1,\hdots,n\}$}\},
$$
with the convention that $W_1 = \CC^2.$ To define a joint signature for algebraic curves under $\mathcal{E}_2(\CC)$, we take $n=4$ and follow a similar construction as the joint signature of smooth curves in $\mathbb{R}^2$ under the action of $\mathcal{E}_2(\mathbb{R})$ (see \cite[Ex. 8.2]{O01}).

\begin{defn}\label{def:jointsig}
The \textit{Euclidean joint signature} of an algebraic curve $C\subset \mathbb{C}^2$ under the action of $\mathcal{E}_2(\CC)$, which we denote $\JS_C$, is the image of the polynomial map $d_4: C^{4}\rightarrow \CC^6$ defined as in \eqref{eq:dmap}.
\end{defn}

\noindent We first show that these invariant functions characterize almost all orbits of the action of $\mathcal{E}_2(\CC)$ on $(\mathbb{C}^2)^{3}$ and $(\mathbb{C}^2)^{ 4}$.

\begin{prop}\label{prop:invsep}
The polynomial invariants $\mathcal{I}_3=\{d_{12},d_{13},d_{23}\}$ separate orbits on $W_3$ for the induced action of $\mathcal{E}_2$ on $(\mathbb{C}^2)^{ 3}$ and the set $\mathcal{I}_4=\{d_{12},d_{13},d_{23},d_{14},d_{24},d_{34}\}$ separates orbits in $W_4$ for the induced action of $\mathcal{E}_2(\CC)$ on $(\mathbb{C}^2)^{ 4}$.
\end{prop}

\begin{proof}
Consider two triples of points $p=(p_i)_{i=1}^3$ and $q=(q_i)_{i=1}^3 \in (\mathbb{C}^2)^{ 3}$, where $p_i=(x_i^p,y_i^p)$ and $q_i$ is denoted similarly, that take the same values on $\mathcal{I}_3$ and lie in $W_3$. Note that $W_3$ excludes isotropic triples such as $(0,0), (1,i), (1,-i).$ We will show that both triples of points necessarily lie in the same orbit. Since $d_{12}\neq 0$ we can choose a representative from the orbit of $p$ under $\mathcal{E}_2$ such that $p_1=(0,0)$ and $p_2=(0,y_2^p)$ by applying the transformation in $\mathcal{E}_2(\CC)$ given by
\begin{equation}\label{eq:movingframe}
c=\frac{y_2^p-y_1^p}{\sqrt{d_{12}}},\, s=\frac{x_2^p-x_1^p}{\sqrt{d_{12}}},\, a=-x_1^p,\, b=-y_1^p,
\end{equation}
and similarly we can assume for $q$ that $q_1=(0,0)$ and $q_2=(0,y_2^q)$. Since $p,q\in W_3$, $y_2^p,y_2^q\neq 0$. Thus $d_{12}(p)=d_{12}(q)$ gives that $(y_2^p)^2=(y_2^q)^2$ meaning $y_2^p=\pm y_2^q$. Therefore, by reflecting about $x$-axis if necessary, we can assume $y_2^p=y_2^q$. The equations $d_{13}(p)=d_{13}(q)$ and $d_{23}(p)=d_{23}(q)$ give
\begin{align*}
(x_3^p)^2+(y_3^p)^2&=(x_3^q)^2+(y_3^q)^2\\
(x_3^p)^2+(y_2^p-y_3^p)^2&=(x_3^q)^2+(y_2^q-y_3^q)^2.
\end{align*}
Subtracting these yields $(y_2^p)^2-2y_2^py_3^p=(y_2^q)^2-2y_2^qy_3^q$ which implies $y_3^p=y_3^q$. Thus, from $d_{13}(p)=d_{13}(q)$, we have $(x_3^p)^2=(x_3^q)^2$. From this we conclude, reflecting about the $y$-axis if necessary, that $x_3^p=x_3^q$. We have now shown that $p$ and $q$ are in the same orbit.

Suppose we have two $4$-tuples of points $p=(p_i)_{i=1}^4$ and $q=(q_i)_{i=1}^4 \in (\mathbb{C}^2)^{3}$ that take the same values on $\mathcal{I}_4$ and lie in $W_4$. By the previous argument we can assume that $p_1,p_2$ have the same form as above and that $p_i=q_i$ for $i=1,2,3$. As before the equations $d_{14}(p)=d_{14}(q)$ and $d_{24}(p)=d_{24}(q)$ imply that and $y_4^q=y_4^p$ and $x_4^p=\pm x_4^q$. If $x_4^p=- x_4^q$ and $x_3^p,x_3^q=0$, then a reflection about the $y$-axis preserves the other values in $q$ and sends $x_4^q$ to $- x_4^q$. Otherwise subtracting the equations $d_{14}(p)=d_{14}(q)$ and $d_{34}(p)=d_{34}(q)$ yields $-2x_3^px_4^p=-2x_3^qx_4^q$, which implies that $x_4^p=x_4^q$. Thus $p$ and $q$ must lie in the same orbit.
\end{proof}

\noindent For any algebraic curve $C$, a generic $n$-tuple of points lies outside of $W_n$. This implies that most points on $C^3$ and $C^4$ lie in the domain of separation for $\mathcal{I}_3$ and $\mathcal{I}_4$.

\begin{lem}\label{lem:genpt}
For an algebraic curve $C\subset\CC^2$ and $n>1$, a generic $n$-tuple of points on $C^n$ lies inside $W_n$. Additionally for any fixed $(n-1)$-tuple of points in $(p_1,\hdots, p_{n-1})\in W_{n-1} \cap C^{n-1}$ and a generic point $p_n\in C$, the $n$-tuple $(p_1,\hdots,p_n)$ lies in $W_n$.
\end{lem}

\begin{proof}
For $n=2$, fix any $p_1 = (x_1, y_1) \in C.$ If $d_{1,2}=0$ for all $(x_2, y_2) \in C,$ then $C$ must lie in a union of lines defined by 
$$
\{ (x_2, y_2) \in \CC^2 \mid (x_1-x_2+iy_1-iy_2) (x_1-x_2-iy_1+iy_2)=0 \}.
$$
Since $C$ is irreducible, this contradicts $\deg(C) > 1.$ Thus the set $
U_{2, p_1} = \{ p_2 \in C \mid d_{1,2} \ne 0 \},$
which is Zariski-open in $C,$ is also nonempty. Thus, for any particular $p_1\in C,$ there exists $p_2$ with $(p_1, p_2) \in W_2 \cap C^2,$ from which both claims follow. Inductively, we fix any $(p_1,\ldots , p_{n-1})\in W_{n-1} \cap C^{n-1}.$ As before, the sets
\[
U_{i, p_1,\ldots p_{n-1} } = \{ p_n \in C \mid d_{i n} \ne 0 \}
\]
are open and nonempty. Thus a generic $p_n \in C$ lies in their intersection, and hence $(p_1,\ldots , p_n)\in W_n.$
%Now suppose we have an $(n-1)$-tuple in $W_{n-1}\cap C^{n-1}$. Then for any $p_n\in C$, the $n$-tuple $p=(p_1,\hdots,p_n)$ lies outside of $W_n$ if and only if $d_{jn}(p)=0$ for some $j\in 1,\hdots, n-1$. For arbitrary $p_n$ and fixed $p_j$, $d_{jn}(p_n)=0$ defines a product of linear spaces in $\mathbb{C}^2$. However, since $\deg(C) >1$ and $C$ is irreducible, a generic point $p_n\in C$ lies outside of these planes, and hence $p\in W_n$. This also implies that if a generic $(n-1)$-tuple of points on $C^{n-1}$ lie in $W_{n-1}$ then a generic $n$-tuple of points on $C^n$ lie in $W_n$. The Lemma follows from induction on $n$.
\end{proof}

\noindent We now characterize the stabilizer groups of points in $W_2$ and $W_3$ under $\E(\CC)$.

\begin{prop}\label{prop:stab}
The stabilizer subgroups of $p \in W_2$ and $p\in W_3$ under the action of $\E(\CC)$ are finite subgroups.
\end{prop}

\begin{proof}
The stabilizer of a point $p\in (\CC^2)^2$ is the subgroup of $\E(\CC)$ given by
$$
\E(\CC)_p = \{g\in \E(\CC)\, |\, g\cdot p=p \}.
$$
The size of the stabilizer of a point is preserved by the action of the group. Since $d_{12}(p)\neq 0$, by applying the transformation in \eqref{eq:movingframe}, we can assume $p$ has the form $p=(p_1,p_2)=(0,0,0,y_2)$ where $y_2\neq 0$. Given the parameterization of $\E(\CC)$ in  Example \ref{Ex:E2Inv}, $g\cdot p=p$ immediately implies that $a=b=0$ and that $sy_2=0$. Thus $\E(\CC)_p$ consists of either the identity transformation or a reflection about the $y$-axis. The same result immediately follows for $p\in W_3$, since $(p_1,p_2,p_3)\in W_3$ implies that $(p_1,p_2)\in W_2$.
\end{proof}

\begin{lem}\label{lem:notdomequiv}
For plane curves $C_0,C_1$, suppose that there exists $p=(p_1,p_2)\in C_0^2, C_1^2$ such that $p\in W_2$ and
$$
d_3(p_1\times p_2 \times C_0) = d_3(p_1\times p_2 \times C_1).
$$
Then $C_0=C_2$ where $C_2 : = h\cdot C_1$ for some $h$ in the stabilizer subgroup of $p$ under $\E(\CC)$.
\end{lem}

\begin{proof}
By Lemma \ref{lem:genpt}, for a generic point $q\in C_0$, the $3$-tuple $(p_1,p_2,q)\in W_3$. Since both curves have the same image under $d_3$, there exists a point $r\in C_1$ such that $r\in d_3^{-1}(p_1,p_2,q)$. By Proposition \ref{prop:invsep}, both triples $(p_1,p_2,q)$ and $(p_1,p_2,r)$ lie in the same orbit under $\E(\CC)$, and hence there exists $g\in \E(\CC)$ such that $g\cdot (p_1,p_2,q)= (p_1,p_2,r)$. However, this implies that $g\in \E(\CC)_{(p_1,p_2)}$. By Proposition \ref{prop:stab}, $\E(\CC)_{(p_1,p_2)}=\{e,h\}$ where $h\in \E(\CC)$ is a reflection about the line containing $p_1$ and $p_2$. Therefore $q=r$ or $h\cdot q = r$, implying that $C_1$ shares infinitely many points with $C_0$ or $h\cdot C_0$, proving the lemma.
%If $\dim(d_3(p_1\times p_2\times C_0)= 0$, since $C_0$ is irreducible, $d_{13}(p_1,p_2,q)$ and $d_{23}(p_1,p_2,q)$ are constant for all $q\in C_0$. As in the proof of Lemma \ref{lem:genpt}, this implies that $\deg(C_0)=1$. Thus 
%$$
%\dim(d_3(p_1\times p_2\times C_0)= \dim(d_3(p_1\times p_2\times C_0)=1,
%$$
%and for generic point $q\in X_0$, there are only finitely many points in $X_0$ in $d_3^{-1}(p_1,p_2,q)$. For each $q$
%\begin{itemize}
%\item dimension of d3(p1p2X0) is one
%
%\item for each q, there are finitely many preimages in $d_3(-1)(p1p2q)$
%
%\item thus X1 shares infinitely many points with X0 or hX0
%\end{itemize}
\end{proof}

\begin{lem}\label{lem:domequiv}
For plane curves $C_0,C_1$, suppose that there exists a $3$-tuple $p=(p_1,p_2,p_3)\in C_0^3, C_1^3$ such that $p\in W_3$ and
$$
d_4(p_1\times p_2 \times p_3\times C_0) = d_4(p_1\times p_2 \times p_3 \times C_1).
$$
Then $C_0=C_2$ where $C_2 : = h\cdot C_1$ for some $h$ in the stabilizer subgroup of $p$ under $\E(\CC)$.
\end{lem}

\begin{proof}
The proof follows similarly as in Lemma \ref{lem:notdomequiv} by applying Propositions \ref{prop:invsep} and \ref{prop:stab}.
\end{proof}

The previous two lemmas provide the basis for showing that the joint invariant signature $\mathcal{J}_C$ characterizes the orbit of $C$ under $\E(\CC)$. Fixing a particular $3$-tuple of points $(p_1,p_2,p_3)\in C^3$, we can look at the slice of $\mathcal{J}_C$ in which $d_4(p_1\times p_2\times p_3\times C)$ lies.

Let $\pi_{ij}: \mathcal{J}_C\rightarrow \CC$ for $1\leq i <j \leq 4$ be the projection of $\mathcal{J}_C$ onto the coordinate $d_{ij}$. Then denote $\mathcal{H}_{ij}(p_1,p_2,p_3,p_4)$ as the linear slice $\pi_{ij}^{-1}(d_{ij}(p_1,p_2,p_3,p_4))\cap \mathcal{J}_C$. Note that for $1 \leq i < j\leq 3$ we can refer to $\mathcal{H}_{ij}(p_1,p_2,p_3)$, as this does not depend on $p_4$. Finally, for any $(p_1,p_2,p_3)\in C^3$ define

$$
\mathcal{Y}_C (p_1,p_2,p_3) := \mathcal{J}_C\cap \mathcal{H}_{12}(p_1,p_2,p_3)\cap \mathcal{H}_{13}(p_1,p_2,p_3)\cap \mathcal{H}_{23}(p_1,p_2,p_3).
$$

Since $d_{ij}$ is an invariant for the action of $\E(\CC)$, all points on $C^3$ that lie in same orbit under $\E(\CC)$ induce the same slice $\mathcal{Y}_C (p_1,p_2,p_3)$. For fixed $(p_1,p_2,p_3)\in C^3$ consider the subset $U\subset C^3$ defined by
\begin{equation}\label{eq:U}
U = \{(\overline{p}_1,\overline{p}_2,\overline{p}_3)\in C^3\,|\, \exists \overline{p}_4\in C\,\, \text{where}\,\, d_4(\overline{p}_1,\overline{p}_2,\overline{p}_3,\overline{p}_4)\in \Y(p_1,p_2,p_3) \}.
\end{equation}
For any $(\overline{p}_1,\overline{p}_2,\overline{p}_3)\in U$, we have that $d_4(\overline{p}_1,\overline{p}_2,\overline{p}_3\times C)\subset \Y(p_1,p_2,p_3)$. Thus we can write

$$
 \Y_C(p_1,p_2,p_3)= \cup_{(\overline{p}_1,\overline{p}_2,\overline{p}_3)\in U} d_4(\overline{p}_1\times\overline{p}_2\times\overline{p}_3\times C).
$$

Since $\Y$ is defined by the values of $d_{12}, d_{13}, d_{23}$ on $(p_1,p_2,p_3)$, by Proposition \ref{prop:invsep}, the set $U$ consists of all points in $C^3$ lying in the same orbit of $(p_1,p_2,p_3)$ under $\E(\CC)$, i.e. $U = \E(\CC)(p_1,p_2,p_3)$.

\begin{prop}\label{prop:Yprop}
Suppose that for two plane curves $C_0, C_1$ such that $\overline{\JS_{C_0}}=\overline{\JS_{C_1}}$ there exists a point $p=(p_1,p_2,p_3,p_4)\in C_0, C_1$ such that
\begin{itemize}
\item $\dim (\Y) = 1$,

\item and $d_4(p)$ is a non-singular point of $\Y$,
\end{itemize}
where $\Y:= Y_{C_0}(p_1,p_2,p_3)=Y_{C_1}(p_1,p_2,p_3)$. Then $C_0=C_2$ where $C_2 : = h\cdot C_1$ for some $h$ in the stabilizer subgroup of $(p_1,p_2,p_3)$ under $\E(\CC)$.
\end{prop}

\begin{proof}
Both sets $d_4(p_1\times p_2\times p_3\times C_0), d_4(p_1\times p_2\times p_3\times C_1)$ are irreducible components of $\Y$. Since $d_4(p)$ is a non-singular point of $\Y$, it necessarily lies in one irreducible component of $\Y$. Thus we have that
$$
d_4(p_1\times p_2\times p_3\times C_0) = d_4(p_1\times p_2\times p_3\times C_1).
$$
The result follows from Lemma \ref{lem:domequiv}.
\end{proof}

Similarly we can define a slice of the smaller variety $d_3(C^3)$,
$$
\Y_C(p_1,p_2) = d_{12}^{-1}(d_{12}(p_1,p_2)) \cap d_3(C^3),
$$
which can be similarly constructed as
$$
\Y_C(p_1,p_2) = \cup_{(\overline{p}_1,\overline{p}_2) \in d_{12}^{-1}(d_{12}(p_1,p_2))} d_3(\overline{p}_1\times\overline{p}_2\times C).
$$

\begin{prop}\label{prop:Ypropd3}
Suppose that for two plane curves $C_0, C_1$ such that $\overline{\JS_{C_0}}=\overline{\JS_{C_1}}$ there exists a point $p=(p_1,p_2,p_3)\in C_0, C_1$ such that
\begin{itemize}
\item $\dim (\Y) = 1$,

\item and $d_3(p)$ is a non-singular point of $\Y$,
\end{itemize}
where $\Y:= Y_{C_0}(p_1,p_2)=Y_{C_1}(p_1,p_2)$. Then $C_0=C_2$ where $C_2 : = h\cdot C_1$ for some $h$ in the stabilizer subgroup of $(p_1,p_2)$ under $\E(\CC)$.
\end{prop}

\begin{proof}
The proof follows similarly as the proof of Proposition \ref{prop:Yprop}, using Lemma \ref{lem:notdomequiv}.
\end{proof}

We consider algebraic curves under two cases: either the image of $C^3$ under the map $d_3: C^3 \rightarrow \CC^3$ is Zariski-dense in $\CC^3$ or not. In either case, we can fix a generic point on $C^4$ and find a slice through this point which satisfies the hypotheses of either Proposition \ref{prop:Yprop} or \ref{prop:Ypropd3}, and hence characterizes the curve.

\begin{lemma}\label{lem:genericP}
Suppose that for a plane curve $C$, the image of the map $d_3 : C^3 \rightarrow \CC^3$ is Zariski-dense in $\CC^3$. Then a generic point $p=(p_1,p_2,p_3,p_4) \in C^4$ satisfies the hypothesis of Proposition \ref{prop:Yprop}.
\end{lemma}

\begin{proof}
Fix a generic $(p_1,p_2,p_3,p_4)\in C^4,$ and consider the subset $U\subset C^3$ defined in \eqref{eq:U}. In particular $U= d_3^{-1}(d_3(p_1,p_2,p_3))$, and hence $U$ has finite cardinality since $d_3(C^3)$ is Zariski-dense in $\CC^3$ \cite[Ch 1, Sec 6.3, Thm 1.25]{Shaf94}. Therefore $\dim(\Y)\leq 1$.

Suppose that $\dim(d_4(p_1\times p_2\times p_3 \times C)) = 0$, then for any $x\in C$, $d_{14}(x) = d_{14}(p_1,p_2,p_3,x)$ is constant on $C$, and similarly so are $d_{24}(x)$ and $d_{34}(x)$. By Lemma \ref{lem:genpt}, $(p_1,p_2,p_3)\notin W_3$ and for a generic $x\in C$, $(p_1,p_2,p_3,x)\notin W_4$. Since $d_{12}(x), d_{24}(x), d_{34}(x)$ are all constant, by Proposition \ref{prop:invsep}, the $4$-tuples $(p_1,p_2,p_3,x)$ are all related by an element of $\E(\CC)$. However each such element lies in the stabilizer subgroup of $(p_1,p_2,p_3)$ which is trivial by Proposition \ref{prop:stab}. This a contradiction as $\dim(C) = 1$, and hence $\dim(d_4(p_1\times p_2\times p_3\times C) =1$ implying that $\dim(\Y)=1.$

Finally we note that $p_4$ is a generic point of $C$, and hence $d_4(p_1,p_2,p_3,p_4)$ is a generic point of the irreducible component $d_4(p_1\times p_2\times p_3\times C)$ of $\Y$. Thus $d_4(p_1,p_2,p_3,p_4)$ is a non-singular point of $\Y$.
\end{proof}

\begin{lemma}\label{lem:genericPd3}
Suppose that for a plane curve $C$, the image of the map $d_3 : C^3 \rightarrow \CC^3$ is \textbf{not} Zariski-dense in $\CC^3$. Then a generic point $p=(p_1,p_2,p_3) \in C^3$ satisfies the hypothesis of Proposition \ref{prop:Ypropd3}.
\end{lemma}

\begin{proof}
Fix a generic $p\in C^4$. We first note that $\dim(d_3(C)) \leq 2$ since $d_3(C)$ is not Zariski-dense in $\CC^3$. Thus if $\dim(\Y_C(p_1,p_2)) =2$, since $d_3(C)$ is irreducible $\overline{d_3(C)} = \overline{\Y_C(p_1,p_2)},$ implying that $d_{12}$ is constant on $C^2$. But for an irreducible curve of degree $>1$, this cannot occur, following a similar argument as in the proof in Lemma \ref{lem:genpt}. Therefore $\dim(\Y_C(p_1,p_2))= 0$ or 1.

Suppose that $\dim(d_4(p_1\times p_2\times C)) = 0$, then for any $x\in C$, $d_{13}(x) = d_{13}(p_1,p_2,p_3,x)$ is constant on $C$, and similarly so is $d_{23}(x)$. Since $d_{13}(x), d_{23}(x)$ are constant, by Proposition \ref{prop:invsep}, the $3$-tuples $(p_1,p_2,x)$ are all related by an element of $\E(\CC)$. However each such element lies in the stabilizer subgroup of $(p_1,p_2,p_3)$ which is of cardinality 2 by Proposition \ref{prop:stab}. This a contradiction as $\dim(C) = 1$, and hence $\dim(d_4(p_1\times p_2\times C) =1$ implying that $\dim(\Y)=1.$

As in the previous proof, $d_3(p_1,p_2,p_3)$ is a generic point of the irreducible component $d_3(p_1\times p_2\times C)$ of $\Y_C(p_1,p_2)$, and hence $d_3(p_1,p_2,p_3)$ is a non-singular point of $\Y_C(p_1,p_2)$.
\end{proof}

\begin{thm}\label{prop:jointsig}
Two plane curves $C_0, C_1\subset \CC^2$ are $\E(\CC)$-equivalent if and only if $\mathcal{J}_{C_0} = \mathcal{J}_{C_1}$.
\end{thm}

\begin{proof}
Note that a generic point of $\JS := \mathcal{J}_{C_0} = \mathcal{J}_{C_1}$ is the image of generic point of $C_0^4$. Fix such a point $p\in C_0^4$. Let $q\in C_1^4$ be a point of $C_1^4$ in the inverse image $d_4^{-1}(d_4(p))$. By Proposition Proposition \ref{prop:invsep} there exists $g\in \E(\CC)$ such that $p = g\cdot q$. Consider the curve $C_2 = g\cdot C_1$.

Suppose $d_3(C_0^3)$ is Zariski-dense in $\CC^3$. Since $p$ is a generic point of $C_0$, by Lemma \ref{lem:genericP} and Proposition \ref{prop:Yprop}, $C_0 = C_2$. In particular $C_0$ is $\E(\CC)$-equivalent to $C_1$. Similarly by Lemma \ref{lem:genericPd3} and Proposition \ref{prop:Ypropd3} if $d_3(C_0^3)$ is not Zariski-dense in $\CC^3$, $C_0$ and $C_1$ are $\E(\CC)$-equivalent.
\end{proof}

The distinction between the cases of when the image of $C^3$ under $d_3$ is Zariski-dense in $\CC^3$ is not arbitrary. We show that this distinction determines whether a curve has an infinite symmetry group under $\E(\CC)$. As in the case of the differential signature, the dimension of the joint signature $\JS_C$ and the cardinality of a fiber of $d_4$ characterize the symmetry group of $C$.

\begin{lemma}\label{lem:noncollinear}
For a plane curve $C$, a generic $3$-tuple $p\in C^3$ has trivial stabilizer subgroup under $\E(\CC)$.
\end{lemma}

\begin{proof}
From the proof of Proposition \ref{prop:stab}, we know that the stabilizer of $p=(p_1,p_2,p_3)$ under $\E(\CC)$ is at most the identity transformation and a reflection about the one-dimensional linear subset of $\CC^2$ defined by $p_1,p_2$. Thus if $p_1,p_2,$ and $p_3$ are non-collinear, the stabilizer subgroup of $p$ is trivial.

Suppose a generic triple on $C^3$ is collinear. Then for a generic tuple $(p_1,p_2)\in C^2$, a generic point of $C$ lies on the linear subset defined by $(x_1-x_2)(y-y_1)+(y_1-y_2)(x-x_1)=0$. This is a contradiction since $C$ is an irreducible plane curve of degree $>1$.
\end{proof}

\begin{proposition}
\label{prop:jointExceptionals}
For an algebraic curve $C$ the following are equivalent.

\begin{itemize}
\item[(1)] The image $d_3(C^3)$ is Zariski-dense in $\CC^3$.

\item[(2)] The symmetry group $|\Sym(C,\E(\CC))| = n < \infty$.

\item[(3)] The map $d_4$ is generically $n:1$.

\item[(4)] $\dim(\JS_C)=4.$
\end{itemize}
\end{proposition}

\begin{proof}

First note that (3) and (4) are equivalent \cite[Ch 1, Sec 6.3, Thm 1.25]{Shaf94}. We now show that (4) and (1) are equivalent. Suppose (1) is true and $\dim(\JS_C) = 3$. Then for a generic point $p\in \JS_C$, $d_3^{-1}(p)$ is finite, but $d_4^{-1}(p)$ is infinite (again by  \cite[Ch 1, Sec 6.3, Thm 1.25]{Shaf94}). This implies that $d_4(p_1\times p_2\times p_3\times C)$ is a single point, which is a contradiction by a similar argument as in Lemma \ref{lem:genericP}. Now suppose that (1) is \textbf{not} true. Then a generic fiber of $d_3$ is infinite, implying the same is true for $d_4$. Thus $\dim(\JS_C) < 4$.

Now suppose (3), that $d_4$ is generically $n:1$, and consider a generic point $p = (p_1,p_2,p_3,p_4) \in C^4$. Then $(p_1,p_2,p_3)$ is generic in $C^3$, and by Lemma \ref{lem:noncollinear} has trivial stabilizer subgroup. By Proposition \ref{prop:invsep} and Lemma \ref{lem:genpt} the fiber $d_4^{-1}(d_4(p))$ is exactly the orbit of $p$ under $\E(\CC)$ intersected with $C^4$, and has cardinality $n$ by assumption. Note that $\Sym(C,\E(\CC))\leq n$, since $p$ has trivial stabilizer subgroup implying that each element of $\Sym(C,\E(\CC))$ sends $p$ to a distinct point on $C^4$.

Consider $q\in d_4^{-1}(d_4(p))$. Then there exists $g\in \E(\CC)$ such that $g\cdot q = p$. Define $\overline {C} : = g\cdot C$. Then $\overline{C}$ and $C$ share the point $p$. Since (3) and (1) and equivalent, by Lemma \ref{lem:genericP} $p$ satisfies the hypothesis of Proposition \ref{prop:Yprop}. Thus there exists $h$ in the stabilizer of $(p_1,p_2,p_3)$ such that $h\cdot \overline{C} = C$. Since the stabilizer subgroup is empty, $h$ is the identity and $\overline{C} = g\cdot C = C$. Thus $g\in \Sym(C,\E(\CC))$. Since $g\cdot q = p$, for each $q\in d_4^{-1}(d_4(p))$ we get a distinct element of $\Sym(C,\E(\CC))$, and hence $|\Sym(C,\E(\CC))| =n$. Thus (3) implies (2).

Finally suppose that the symmetry group of $C$ is of finite cardinality. If (1) is true, then we are done. Thus suppose that the image $d_3(C^3)$ is \textbf{not} Zariski-dense in $\CC^3$, and consider a generic point $p= (p_1,p_2,p_3) \in C^3$. By Proposition \ref{prop:invsep} and Lemma \ref{lem:genpt} the fiber $d_3^{-1}(d_3(p))$ is exactly the orbit of $p$ under $\E(\CC)$ intersected with $C^3$. Furthermore by Lemma \ref{lem:genericPd3} $p$ satisfies the hypothesis of Proposition \ref{prop:Ypropd3}. Note that the fiber $d_3^{-1}(d_3(p))$ is of infinite cardinality \cite[Ch 1, Sec 6.3, Thm 1.25]{Shaf94}.

Consider $q\in d_3^{-1}(d_3(p))$. Then there exists $g\in \E(\CC)$ such that $g\cdot q = p$. Define $\overline {C} : = g\cdot C$. Then $\overline{C}$ and $C$ share the point $p$. By Proposition \ref{prop:Ypropd3}, there exists $h$ in the stabilizer of $(p_1,p_2)$ such that $h\cdot \overline{C} = C$. Thus $hg\in \Sym(C,\E(\CC))$. Each distinct $q\in d_3^{-1}(d_3(p))$ yields an element $h_qg_q \in \Sym(C,\E(\CC))$. Note that since each $q$ is distinct, each $g_q$ is distinct. Since $h_q$ is in the stabilizer of $(p_1,p_2)$ which is of cardinality 2 (see proof of Lemma \ref{lem:noncollinear}), this implies that there are infinitely many elements of $\Sym(C,\E(\CC))$, which is a contradiction. Thus (2) implies (1).
\end{proof}

\noindent The next proposition immediately follows.

\begin{proposition}
For an algebraic curve $C$ the following are equivalent.

\begin{itemize}
\item[(1)] The image $d_3(C^3)$ is \textbf{not} Zariski-dense in $\CC^3$.

\item[(2)] The symmetry group $|\Sym(C,\E(\CC))|$ is infinite.

\item[(4)] $\dim(\JS_C)<4$
\end{itemize}
\end{proposition}

We end this section with a discussion about joint signatures for other algebraic group actions. In \cite{O01}, the author presents a \textit{smooth} characterization of joint invariants for many of these groups over $\RR$. For instance, consider the action of $\SA(\RR)$ on $n$-tuples of $\RR^2$. The fundamental joint invariant are given by the signed area functions $v(i,j,k)$ for $1\leq i < j < k \leq n$ (\cite[Thm 3.3]{O01}) where
$$
v(i,j,k) = x_i(y_j-y_k) - x_j(y_i-y_k) + x_k(y_i-y_j).
$$
Though the number of such invariants increases in size rapidly as $n$ grows, there exists many linear syzygies between these functions. In particular the invariants $v(1,2,k), v(1,3,k)$ for $k=2,\hdots, n$ generate the other invariants \cite[Thm 8.8]{O01}. Thus for curves under $\SA(\CC)$ we can define the map $v_6: (\CC^2)^6 \rightarrow \CC^7$ by
\begin{equation}\label{eq:EqAffSig}
(x_1,y_1,\hdots,x_6,y_6)\mapsto  (v(1,2,3), v(1,2,4), v(1,2,5), v(1,2,6), v(1,3,4), v(1,3,5), v(1,3,6)).
\end{equation}
Mirroring the construction for real curves under $\SA(\RR)$ in \cite[Ex. 8.6]{O01}, we can define the \textbf{equi-affine joint signature} of a curve $C$ to be $\JA_C = v_6(C^6)$. Though there are fundamental invariants for $n$ as low as 3, it is necessary to consider 6-tuples of points on curves since all curves have the same image under the map $v_n$ (defined as above) when $n < 6$. For other groups, the fundamental joint invariants presented in \cite{O01} similarly yield sets of algebraic invariants. It would be interesting to construct general conditions for sets of joint invariants to characterize orbits of curves.

\begin{remark}\label{rem:EqAff}
While we conduct experiments in Section \ref{sec:implementation} comparing the equi-affine joint signature to the equi-affine differential signature, we do not explicitly prove that $\JA_C$ characterizes orbits of curves under $\SA(\CC)$, as we do for the Euclidean joint signature. However, as the seven area invariants defining $v_6$ generate the other fundamental area invariants through linear relations, it is likely that they separate orbits and that one can prove this characterization using an argument similar to that in this section.
\end{remark}

\section{Witness sets for signatures}\label{sec:nag}

\subsection{Background}\label{sec:nagbackground}

\label{subsec:nag}

A comprehensive overview of numerical algebraic geometry may be found in the survey~\cite{SVW05} or books~\cite{SW05,BHSW13}.
Here we develop the notions that we need, illustrated by several examples related to the previous section.

The main data structures in numerical algebraic geometry are variations on the notion of a \emph{witness set.}
The overarching idea is to represent an irreducible variety $Y\subset \CC^n $ by its intersection with a generic affine linear subspace of complementary dimension.
The number of points in such an intersection is the degree $\deg Y,$ which may be understood as the degree of the projective closure of $Y$ under the usual embedding $\CC^m \ni (x_1, \ldots, x_m) \mapsto [x_1 : \cdots : x_m : 1] \in \PP (\CC^{m+1}).$ We define a $c$-slice in $\CC^m$ to be a polynomial system consisting of $c$ affine hyperplanes, $L = (l_1, \ldots , l_c)$ with $l_i \in \CC [x_1, \ldots , x_m]_{\le 1}.$ For convenience we write $L$ in place of $V(L(x))$ and also use the notation $L^c.$ For $Y$ an irreducible variety of codimension $c$ and a generic slice $L^c,$ the intersection $Y\cap L^c$ is \emph{transverse}, consisting of $\deg Y$ isolated, nonsingular points.

The notion of a \emph{pseudo-witness set}, first appearing in~\cite{HS10}, allows us to represent the closed image of a rational map $Y = \imclo{\Phi}$ without knowing its implicit defining equations. Our Definition~\ref{def:witness-set} differs slightly from that used in the standard references~\cite{HS10,HS13,BHSW13}; to distinguish our setup, we provisionally use the term \emph{weak pseudowitness set.} 
\begin{defn}
  \label{def:witness-set}
Let $V(f)\subset \CC^n$ be Zariski-closed, $X\subset V(f)$ be one of its irreducible components, and $\Phi : X \ratmap \CC^m$ be a rational map. Set $c= \codim V(f),$ $d= \dim \imclo{\Phi}.$ A weak pseudowitness set for $\Phi$ is a quadruple $\left( f, \Phi, (L|L'), \{ w_1, \ldots , w_e \} \right),$ where $L$ is a generic affine $(m-d)$-slice of $\imclo{\Phi },$ $L'$ is a generic affine $(c-m+d)$-slice of $X,$ and such that $w_1, \ldots , w_e$ are points in $X \cap L'$ where $\Phi $ is defined such that $\imclo{\Phi} \cap L = \{ \Phi (w_1), \ldots , \Phi (w_e) \}$ and $e=\deg \imclo{\Phi}.$
\end{defn}

\begin{example}
\label{example:ellipse}
Consider again the ellipses from Example~\ref{Ex:E2Inv}. We represent $Y=\imclo{\sigma_{C_1}} =\imclo{\sigma_{C_2}}$ not by the signature polynomial, but rather by its intersection with a generic slice in the codomain: $L^1 = \{ l_1 x + l_2 y + l_3 = 0 \}.$ For the particular choice of $(l_1,l_2,l_3) = (1,-2,1),$ we have that $\imclo{\sigma_{C_i}} \cap  L^1$ consists of $6$ points $(x_i, y_i)\in \CC^2$:
\begin{center}
\begin{tabular}{ccc}$x_i$ & & $y_i$\\\hline$.0305676-.0677494\, i $ & &
      $.515284-.0338747\, i$\\$-.120636-.0158199\, i $ & & $.439682-.00790993\,
      i$\\$-.120636+.0158199\, i $ & & $.439682+.00790993\, i$\\$.501814 $ & &
      $.750907$\\$.0305676+.0677494\, i $ & & $.515284+.0338747\, i$\\$4.17832
      $ & & $2.58916$\end{tabular}
\end{center}
\end{example}

Compared to the more standard definition of a witness set, in Definition~\eqref{def:witness-set} we allow that the containment $\{ w_1 , \ldots , w_e \} \subset L' \cap \Phi^{-1} \left( \imclo{\Phi} \cap L\right)$ may be proper.
When $\Phi $ is the signature of a curve with many symmetries, this may be preferable, since fewer points need to be stored
due to Theorem~\ref{thm:diffsigsymgroup} and Proposition~\ref{prop:jointExceptionals}.
%% four leaf clover
%% and its signature
%% continue with teaser example
The data in Definition~\ref{def:witness-set} are already sufficient for testing queries of the form $y\in \imclo{\Phi},$ as noted in~\cite[Remark 2]{HS10}. For testing, $y\in \im \Phi $ and other applications, the stronger notion is required~\cite{HS13}. Further applications of pseudowitness sets may be found in the references~\cite{CK19,Bry18,Reganpseudo,Sottilepseudo}. 

In our context, equations defining $\imclo{\Phi}$ are seldom known, so in what follows we may informally refer to the objects of Definition~\ref{def:witness-set} and their multiprojective counterparts in Definition~\ref{def:multi-witness-set} as ``witness sets'' without ambiguity. In practice, we can at best hope that our numerical approximations to points $\Phi(w_1), \ldots , \Phi(w_e)$ lie sufficiently close to $\imclo{\Phi} \cap L$: to clearly distinguish practice from theory, we occasionally use the term \emph{numerical (weak / pseudo) witness set.} 

Following~\cite{HR15,LRS18,HLRS}, we give a multiprojective generalization of Definition~\ref{def:witness-set}. For irreducible $Y\subset \CC^m,$ we fix $(m_1,\ldots , m_k),$ an integer partition of $m,$ and consider $Y$ in the affine space $\CC^{m_1}\times \cdots \times \CC^{m_k}.$ We consider slices $L^{\bs{e}}=L^{e_1}|\cdots |L^{e_k},$ where $\bs{e} = (e_1,\ldots , e_k)\in \NN^k$ is an integral vector such that $e_1+\cdots e_k= \dim Y,$ and $L^{e_j}$ is a $e_j$-slice consisting of $e_j$ affine hyperplanes in the coordinates of $\CC^{m_j}.$ We say that $\bs{e}$ is a \emph{multidimension} of $Y$ if for generic $L^{\bs{e}}$ the intersection $Y \cap L^{\bs{e}}$ is a finite set of nonsingular points; the number of points for such $L^{\bs{e}}$ is a constant called the \emph{$\bs{e}$-multidegree} $\deg_{\bs{e}} Y.$
These definitions reflect the geometry of the \emph{multiprojective closure} of $X$ under the emebedding
\begin{displaymath}
Y \ni (y_1,\ldots , y_m) \mapsto
\Big( [y_1:\cdots : y_{m_1} : 1] , \, \ldots .
%\end{displaymath}
%\begin{displaymath}
[y_{m- m_k +1} : \cdots : y_{m} : 1 ] \Big) \in \PP (\CC^{m_1+1}) \times \cdots \times \PP (\CC^{m_k+1}).
\end{displaymath}
\begin{defn}
  \label{def:multi-witness-set}
Let $f, X, c, L', \Phi$ be as in~\ref{def:witness-set}, and $\bs{e}$ be a multi-dimension of $\imclo{\Phi}$ corresponding to some partition of $n.$ An $\bs{e}$-weak pseudowitness set for $\Phi$ consists of $\big( f, \Phi, (L^{\bs{e}}|L'),$ $\{ w_1, \ldots , w_e \} \big),$ such that $\imclo{\Phi} \cap L^{\bs{e}} = \{ \Phi (w_1), \ldots , \Phi (w_e) \}$ and $e=\deg_{\bs{e}} \imclo{\Phi}.$
\end{defn}
\begin{example}
\label{example:ellipseMulti}
Continuing as in Example~\ref{example:ellipse}, we now consider coordinate slices in the codomain of $\sigma_{C_i}$ of the form $L^{(1,0)} =\{ l_1 x + l_2 y + l_3 = 0 \}.$ Specializing to the generic slice $(l_1,l_2) = (3, 1)$ yields now $3$ points:
\begin{center}
\begin{tabular}{ccc}$x_i$ & & $y_i$\\\hline$-.333333 $ & &
      $1.53234+1.11277\, i$\\$-.333333 $ & & $1.53234-1.11277\, i$\\$-.333333 $
      & & $-6.06468$\end{tabular}
\end{center}
\end{example}

The general membership test for multiprojective varieties proposed in~\cite{HR15} uses the stronger notion of a witness collection. This is required since for an arbitrary point $y\in Y$ there may not exist transverse slices $L^{\bs{e}}\ni y$ for $\bs{e}$ ranging over all multidimensions of $Y$---see~\cite[Example 3.1]{HR15}. This subtlety is not encountered for generic $y\in Y$; we record this basic fact in Proposition~\ref{prop:Le}.
\begin{prop}
\label{prop:Le}
Fix irreducible $Y\subset \CC^{m_1} \times \cdots \times\CC^{m_k}$ and $\bs{e}$ some multi-dimension of $Y.$ For $y = (y_1,\cdots , y_k) \in Y$ generic,
there exists an $\bs{e}$-slice $L^{\bs{e}}\ni y$ such that $\dim (Y \cap L^{\bs{e}})=0.$ Moreover, for $y\notin Y_{sing},$ we also have that $y \notin (Y \cap L^{\bs{e}})_{sing}$ for generic $L^{\bs{e}}.$
\end{prop}
\begin{proof}
For generic $y_1$ in the image of $\pi_1: Y \to \CC^{m_1}$ we have that the fiber $\pi_1^{-1} (x_1)$ has dimension $\dim Y-\dim \pi_1(Y).$ Choose such an $y_1$ and let $L^{e_1}\ni y_1$ be generic so that $\pi_1 (Y) \cap L^{e_1}$ has dimension $\dim \pi_1(Y)-e_1.$ It follows that $Y \cap L^{e_1}$ has dimension $\dim Y-e_1.$ This construction holds for all $y_1$ on some Zariski open $U_1\subset \pi_1 (Y).$ Repeating this construction for the remaining factors yields $U_2,\ldots , U_k$ such that the first part holds for all $y\in U_1\times \cdots  \times U_k.$ The second part follows from Bertini's theorems, eg.~\cite[Thm 17.16]{Har13}.
\end{proof}

\subsection{A general equality test}\label{sec:equality}

Now let $\Phi_0 : X_0 \ratmap \CC^m$ and $\Phi_1: X_1 \ratmap \CC^m$ denote two rational maps with each $X_i \subset \CC^{n_i}$ of codimension $c_i.$ Problem~\ref{problem:map} from the introduction asks us to decide whether or not their images are equal up to Zariski closure. We describe a probabilistic procedure )Algorithm~\ref{alg:equality-test}) which refines the general membership and equality tests from numerical algebraic geometry, which are summarized in~\cite[Ch.~13, 15]{SW05} and~\cite[Ch.~8,16]{BHSW13}. As noted in the Introduction, our setup is motivated by an efficient solution to Problem~\ref{problem:group}. Following the standard terminology, our test correctly decides equality with ``probability-one'' in an idealized model of computation. This is the content of Theorem~\ref{thm:correctness}. Standard disclaimers apply, since any implementation must rely on numerical approximations in floating-point. A thorough discussion of these issues may be found in \cite[Ch. 3, pp. 43-45]{BHSW13}.
%The actual success rate depends on the typical conditioning of various subproblems, the amount of precision used, implementation-specific evaluation and approximation schemes, and many other factors.% \textcolor{red}{For more discussion on this see \cite{BHSW13}}

Algorithm~\ref{alg:equality-test} assumes different representations for the two maps. The map $\Phi_1$ is represented by a witness set in the sense of Definition~\ref{def:witness-set}, say $(f_1, \Phi_1, (L_1|L_1'),\{ w_1, \ldots , w_e \}).$ In fact, the only data needed by Algorithm~\ref{alg:equality-test} are the map itself $\Phi_1,$ the slice $L_1,$ and the points $w_1,\ldots , w_e.$ For the map $\Phi_0,$ we need only a sampling oracle that produces generic points on $X_0$ and $\codim (X_0)$-many reduced equations vanishing on $X_0.$
%Despite being technical, these assumptions may be relevant in situations where we can sample from $X_0$ (eg.~via a parametrization) but the ideal $\mathcal{I}_{X_0}$ is unknown. We sketch the multiprojective generalization of Algorithm~\ref{alg:equality-test} at the end of the section.

%\textcolor{red}{The dimension of $\imclo{\Phi_1}$, denoted $d$, can be determined from the witness set; to check if $\dim \imclo{\Phi_0}=d$, we compute .... Now S}
Suppose $\dim \imclo{\Phi_0} = \dim \imclo{\Phi_1} = d.$ There is a probabilistic membership test for queries of the form $\Phi_0(x_0) \in \imclo{\Phi_1 }$ based on homotopy continuation. The relevant homotopy depends parametrically on $L_1,$ a $(m-d)$-slice $L_0\ni \Phi_0 (x_0),$ a $(c_0-m+d)$-slice $L_0'\ni x_0,$ and a regular sequence $f_0=(f_{0,1},\ldots , f_{0,c_0})$ which is generically reduced with respect to $X_0.$ The homotopy $H$ is defined by setting
\begin{equation}
\label{eq:homotopy}
H(x;t) = \left(
\begin{array}{c}
f_0 (x)\\
L_0' (x)\\
t \, L_1 \circ \Phi_0 + (1-t)\, L_0 \circ \Phi_0  (x)
\end{array}\right) =0.
\end{equation}
In simple terms, $H$ moves a slice through $\Phi_0 (x_0)$ to the slice witnessing $\imclo{\Phi_1}$ as $t$ goes from $0$ to $1.$ A solution curve associated to~\eqref{eq:homotopy} is a smooth map $x:[0,1]\to \CC^n$ such that $H(x(t),t)=0$ for all $t.$ For generic parameters $L_0, L_1, L_0'$ the Jacobian $H_x (x,t)$ is invertible for all $t\in [0,1],$ solution curves satisfy the ODE
\begin{displaymath}
x'(t) = - H_x (x,t)^{-1} H_t (x,t),
\end{displaymath}
and each of the points $w_1,\ldots ,w_e$ is the endpoint of some solution curve $x$ with $x(0) \in X \cap L_0'.$ These statements follow from more general results on \emph{coefficient-parameter homotopy}, as presented in~\cite{MS89} or~\cite[Thm 7.1.1]{SW05}. We assume a subroutine $\TRACK (H,x_0)$ which returns $x(1)$ for the solution curve based at $x_0.$ In practice, the curve $x(t)$ is approximated by numerical predictor/corrector methods~\cite{AG12, Mor09}. We allow our $\TRACK$ routine to fail; this will occur, for instance, when $\Phi_0 (x_0)$ is a singular point on $\imclo{\Phi_0 }.$ However, it will succeed for generic (and hence \emph{almost all}) choices of parameters and $x_0 \in \CC^{n_0}.$ %Algorithm~\ref{alg:equality-test} exploits this fact.
\begin{figure}
\begin{algorithm}{Probability-1 equality test}
\label{alg:equality-test}
\begin{algorithmic}[1]
    \smallskip\hrule\smallskip
    \REQUIRE Let $X_0\subset \CC^{n_0} ,X_1 \subset  \CC^{n_1}$ be irreducible algebraic varieties, and $\Phi_0 : X_0 \to \CC^m,$ $\Phi_1 : X_1 \to \CC^m$ be rational maps, represented via the following ingredients:
  \begin{itemize}
    \setlength\itemsep{.25em}
  \item[1)] $(L_1, \{ w_1, \ldots , w_e \})$ with $\imclo{\Phi_1} \cap L_1=\{ \Phi_1 (w_1),\ldots , \Phi_e (w_e) \}$ and $e= \deg \imclo{\Phi_1}$
  (cf.~Definition~\ref{def:witness-set}),\\
  \item[2)] $f_{0,1},\ldots , f_{0,c_0} \in \CC[x_1,\ldots , x_{n_0}]$: a generically reduced regular sequence such that $\codim (X_0) = c_0$ and $X_0 \subset V(f_1,\ldots , f_{c_0}),$
  \item[3)] an oracle for sampling a point $x_0 \in X_0,$ and\\
  \item[4)] explicit rational functions representing each map $\Phi_i.$
  \end{itemize}
  \ENSURE YES if $\imclo{\Phi_0} = \imclo{\Phi_1}$ and NO if $\imclo{\Phi_0} \ne \imclo{\Phi_1}.$
  \STATE sample $x_0 \in X_0$
\STATE $T_{x_0} (f) \gets \ker \, \diff{f}{x_0}$
\STATE $d\gets \rk \, \diff{\Phi_0}{x_0}\big|_{T_{x_0} (f)}$
\STATE \textbf{if} $d\ne \dim \imclo{\Phi_1}$ \textbf{then} \textbf{return} NO
\STATE $H(x;t) \gets$ the homotopy from equation~\ref{eq:homotopy}
  \STATE $x_1 \gets \TRACK \left(H, x_0\right)$
  \STATE \textbf{if} $\Phi_0(x_1) \in \{ \Phi_1 (w_1), \ldots , \Phi_1(w_e) \}$ \textbf{return} YES\\
  \indent \, \, \, \textbf{else return} NO
\smallskip\hrule
  \end{algorithmic}
  \end{algorithm}
\caption{A general, probabilistic equality test for rational maps.}
\end{figure}
\begin{thm}
\label{thm:correctness}
For generic $x_0, L_0, L_0',L_1,$ Algorithm~\ref{alg:equality-test} correctly decides if $\imclo{\Phi_0}=\imclo{\Phi_1}.$
\end{thm}
\begin{remark}
% To apply the coefficient-parameter theory to the homotopy $H,$
The set of ``non-generic'' $L_1$ depends on $\Phi_0$ and $\Phi_1.$ In practice, an oracle for sampling generic points could be provided by either a parametrization or by homotopy continuation with known equations for $X_0.$ The dimension $\dim \imclo{\Phi_1}$ is implicit in the description of the witness set.
% for fixed $(\Phi_1, L_1),$ a malicious adversary could cook up $\Phi_0$ with $\dim \imclo{\Phi_0}=\dim \imclo{\Phi_1}$ and $\dim (\imclo{\Phi_0} \cap L_1) > 0.$ In this case, we assume Algorithm~\ref{alg:equality-test} fails at line 6. We could also use a homotopy similar to $H$ to compute a new witness set for $\Phi_1$ beforehand.
\end{remark}
\begin{proof}
Since $x_0$ is generic and $f_0$ is generically reduced, we may assume that that $d=\dim \imclo{\Phi_0}.$ Noting line 4, we are done unless $d=\dim \imclo{\Phi_1}.$ In this case, since the $\imclo{\Phi_i}$ are irreducible,
\begin{equation}
\label{eq:dimdrop}
\dim \left( \imclo{\Phi_0} \cap \imclo{\Phi_1} \right) = d
\, \, \, \, 
\Leftrightarrow
\, \, \, \, \, \, 
\imclo{\Phi_0} = \imclo{\Phi_1}.
\end{equation}
As previously mentioned, generic slices give that the solution curve $x(t)$ associated to~\ref{eq:homotopy} with initial value $x_0$ exists and satisfies $x(t) \in V(f) \setminus V(f)_{\sing}$ for all $t\in [0,1].$ The endpoint $x_1$ is, \emph{a priori}, a point of $V(f).$ Since $X_0\setminus (X_0)_{\sing}$ is a connected component of $V(f) \setminus V(f)_{\sing}$ in the complex topology and $x_0 \in X_0,$ so also must $x_1\in X_0.$ Hence $\Phi_0 (x_1) \in \imclo{\Phi_0} \cap L_1.$ Now if $\imclo{\Phi_0} = \imclo{\Phi_1},$ then clearly we must have
\begin{equation}
\label{eq:witness}
\Phi_0(x_1) \in \imclo{\Phi_1} \cap L_1 = \{ \Phi_1 (w_1), \ldots , \Phi_1 (w_e) \},
\end{equation}
as is tested on line 7. Conversely, if~\eqref{eq:witness} holds, then
\begin{displaymath}
\dim (\imclo{\Phi_0} \cap \imclo{\Phi_1} \cap L_1)\ge 0,
\end{displaymath}
which by~\eqref{eq:dimdrop} and the genericity of $L_1$ implies $\imclo{\Phi_0} = \imclo{\Phi_1}.$
\end{proof}
In the multiprojective setting, we may give a similar argument. The only added subtley is that extra genericity may be needed so that the Jacobian $H_x (x_0,0)$ is invertible. This follows from Proposition~\ref{prop:Le}.

\subsection{Witness sets for signatures}
\label{subsec:witness}

Our implementation of Algorithm~\ref{alg:equality-test} treats only the special case where the domain of each rational map is some Cartesian product of irreducible plane curves, say $X_i = C_i^k$ for some integer $k.$ For the purpose of our implementation, the various ingredients for the input to Algorithm~\ref{alg:equality-test} are easily provided. Suppose $\mathcal{I}_{C_i}=\langle f_i\rangle$ for $i=0,1.$ Then the reduced regular sequence we need is given by $\left( f_0 (x_1,y_1), \ldots , f_0 (x_k,y_k) \right).$ Sampling from $X_0$ amounts to sampling $k$ times from $C_0$; we sample the curve $C_0$ using homotopy continuation from a linear-product start system~\cite[Sec 8.4.3]{SW05}.\\\\
It remains to discuss computation of the witness set for the image of the signature map.
We now summarize the relevant techniques from numerical algebraic geometry in this setting.
For a \emph{generic plane curve} of degree $d,$
\[
f(x,y;p) = p_{0,0} + p_{1,0} x + p_{0,1} y + p_{1,1}  x\, y + \cdots + p_{0,d} y^d,
\]
it is natural to consider the signature map which is a rational function in the parameters $p=(p_{0,0}, \ldots , p_{0,d}).$
We may write the parametric signature map as $\Phi (x,y;p) = (\Phi_1 (x,y;p), \ldots, \Phi_m (x,y;p)).$
There is an associated incidence correspondence
\[
V_\Phi = \{ (x_1,y_1,\ldots , x_k, y_k, p, L) \in (\mathbb{C}^2)^k \times \mathbb{C}^{\binom{d+2}{2}} \times \mathbb{G}_{k,m} \mid f(x_i, y_i, p) = 0 , \Phi (x,y; p) \in L \},
\]
where $\mathbb{G}_{k,m}$ denotes the Grassmannian of codimension-$k$ affine subspaces of $\CC^m.$ For generic $L,$ the fiber over $(p, L)$ of the projection $\pi: V_\Phi \rightarrow \mathbb{C}^{\binom{d+2}{2}} \times \mathbb{G}_{k,m}$ is naturally identified with a pseudowitness set for the signature map of the curve corresponding to $p$, denoted $\Phi (\cdots  ; p).$

\begin{proposition}\label{prop:VPHI}
The incidence variety $V_\Phi$ is irreducible.
\end{proposition}

\begin{proof}
The fibers of the coordinate projection $\pi_1: V_{\Phi} \to (\CC^2)^k \times \CC^{\binom{d+2}{2}}$ given by $(x,y,p, L) \mapsto (x,y,p)$ are affine-linear spaces.
Thus, a Zariski-open subset of $V_\Phi $ is an affine bundle over some base $B_1\subset (\CC^2)^k\times \CC^{\binom{d+2}{2}}.$ 
Let $X$ denote the Zariski closure of the image of $\pi_1.$ 
We may then consider the coordinate projection $\pi_2: X \to \CC^{\binom{d+2}{2}},$ whose image is Zariski-dense in $\CC^{\binom{d+2}{2}}.$
Once again, $\pi_2$ restricts to an affine bundle over a base $B_2 \subset \CC^{\binom{d+2}{2}},$ which is now easily seen to be connected in the complex topology. 
Therefore, considering both $\pi_2$ and $\pi_1,$ there exists a dense Zariski-open subset $U \subset V_\Phi$ which is connected in the complex topology.
Therefore $U$ is irreducible in the Zariski topology, giving that $\overline{U} = V_\Phi $ is also irreducible.
\end{proof}

The fiber $\pi^{-1}(p,L)$ of the projection defined above is a zero-dimensional subset of $V_\Phi$ with cardinality $N$ given by the product of degree of the signature variety and the size of the symmetry group for a generic curve of degree $d$. Consider the subset $B\subset \pi(V_\Phi)$ where each fiber has the same cardinality, i.e. \[
B = \{(p,L)\, :\, |\pi^{-1}(p,L)| = N\}.
\]
Note that a generic $(p,L)\in\mathbb{C}^{\binom{d+2}{2}} \times \mathbb{G}_{k,m}$ lies in $B$.
For a fixed $(p,L)\in B$, the \emph{monodromy group} $\calM (\pi ; p,L)$ is a permutation group which acts on the fiber of $\pi^{-1}(p,L)$ by lifting loops in $B$ based at $(p,L)$ to $V_\Phi$.
Proposition \ref{prop:VPHI} acts transitively on the points in the pseudowitness set.
It now follows that, for the curve of interest given by $p_1\in \CC^{\binom{d+2}{2}},$ we may compute a pseudo-witness set for the signature $\Phi_1 = \Phi (\cdots ; p_1)$ using the following steps, which are standard in numerical algebraic geometry:
\begin{itemize}
\item[1)] Fix generic $(x_0,y_0)\in \CC^{2k}$, and find $(p_0, L_0)$ so that $(x,y,p_0, L_0) \in V_\Phi$ by solving linear systems of equations.
\item[2)] Using the transitivity of the monodromy group, complete $(x_0,y_0)$ to a pseudowitness set for the curve given by $p_0.$
\item[3)] The pseudowitness set for $p_1$ will consist of finite endpoints as $t\to 1$ of the homotopy (see.~\cite{MS89})
\begin{equation}
\label{eq:homotopy2}
H_p(x;t) = \left(
\begin{array}{c}
f(x,y; t p_1 + (1-t) p_0 )\\
L \circ \Phi (x, y; t p_1 + (1-t) p_0)
\end{array}\right) =0.
\end{equation}
\end{itemize}

Viewed as a subgroup of the symmetric group $S_N,$ it natural to ask how the monodromy group $\calM (\Phi ; p, L)$ depends on the type of signature map $\Phi $ and the generic curve degree $d.$
As soon as $d$ is large enough, a generic curve specified by $p$ will have a trivial symmetry group.
In our experiments, our computations show that each monodromy group $\calM (\Phi ; p, L)$ is the entire symmetric group $S_N$ in such cases.
In general, we have that $\calM (\Phi ; p, L)$ is a subgroup of the \emph{wreath product} $S_{N_1} \wr S_{N_2},$ where $N_1$ is the size of the generic symmetry group and $N_2 = \deg \imclo{\Phi_{p}}.$
Thus, for families of curves with a nontrivial symmetry group,  $\calM (\Phi ; p, L)$ is \emph{imprimitive}, in which case the \emph{decomposable monodromy} technique from~\cite{AR16} may be used to speed up witness set computation.
Finally, we note that Proposition~\ref{prop:VPHI} and the homotopy from Equation~\eqref{eq:homotopy2} can be considered in more structured settings---for instance, when the curves $f(x,y;p)$ are drawn from a linear subspace of $p\in \CC^{\binom{d+2}{2}},$ or when the image slices $L$ are multiprojective as in the sense of Definition~\ref{def:multi-witness-set}.
We leave the study of monodromy groups in these settings as an interesting direction for further research.

\section{Implementation, examples, and experiments}\label{sec:implementation}

Our results showcase features of the \verb|NumericalAlgebraicGeometry| ecosystem in Macaulay2 (aka \verb|NAG4M2|, see~\cite{Ley11, Ley18} for an overview.) We rely extensively on the core path-tracker and the packages \verb|SLPexpressions| and \verb|MonodromySolver|. All of our examples and experiments deal with differential and joint signatures for either the Euclidean or equi-affine group.\footnote{For details we refer to the code: \url{https://github.com/timduff35/NumericalSignatures.}} However, the current functionality should make it easy to study other group actions and variations on the signature construction in the future. 

The differential signatures for curves under $\E(\CC)$ and $\SA(\CC)$ are defined in Examples \ref{Ex:E2Inv} and \ref{ex:EqAffDiff} respectively, and the joint signatures are defined in Definition \ref{def:jointsig} and in \eqref{eq:EqAffSig}. To distinguish between the two groups, for a curve $C$, we denote the Euclidean differential and joint signatures of $C$ as $\JE_C$ and $\SigE_C$ respectively. Similarly we denote the equi-affine differential and joint signatures of $C$ as $\JA_C$ and $\SigA_C$. We caution that we do not explicitly prove that $\JA_C$ characterizes the equivalence class of $C$ under $\SA(\CC)$, as we do for the Euclidean joint signature. However as we explain in Remark \ref{rem:EqAff} it is likely that it does.

We explain some aspects of our implementation that appear to give reasonable numerical stability. A key feature is that polynomials and rational maps are given by straight-line programs as opposed to their coefficient representations. This is especially crucial in the case of differential signatures, where we can do efficient evaluation using the formulas in equation~\ref{eq:SLP}; we note that expanding these rational functions in the monomial basis involves many terms and does not suggest a natural evaluation scheme. We also homogenize the equations of our plane curves and work in a random affine chart. Finally, in our sampling procedure we discard samples which map too close to the origin in the codomain of our maps, as these tend to produce nearly-singular points on the image.
\begin{example}
\label{ex:quartics}
The code below computes a witness set for the Euclidean differential signature of a ``generic'' quartic (whose coefficients are random complex numbers of modulus 1.)
\begin{verbatim}
(d, k) = (4, 1);
dom = domain(d, k);
Map = diffEuclideanSigMap dom;
H = witnessHomotopy(dom, Map);
W = runMonodromy H;
\end{verbatim}
To compute a witness set for the differential signature of the Fermat quartic $V( x^4+y^4+z^4)\subset \PP (\CC^3),$ we use the previous computation.
\begin{verbatim}
R = QQ[x,y,z];
f=x^4+y^4+z^4;
Wf = witnessCollect(f, W)
\end{verbatim}
The output resulting from the last line reads
\begin{verbatim}
witness data w/ 18 image points (144 preimage points)
\end{verbatim}
indicating that the Euclidean differential signature map is generically 8 to 1, which is equivalent to the Fermat curve having eight Euclidean symmetries \cite[Thm 2.38]{KRV18}. We timed these witness set computations at $5$ and $0.5$ seconds, respectively. For joint signatures, the analagous computations were timed at $95$ and $17$ seconds. 
\end{example}
\begin{figure}
  \begin{tabular}{c|c|c|c|c}
$d$ & $\deg \SigE$ & time (s) & $\deg_{(1,0)}\SigE$ & time (s) \\
\hline
2 & 6 & 0.3 & 3 & 0.1 \\
3 & 72 & 2 & 36 & 0.5\\
4 & 144 & 9 & 72 & 2\\
5 & 240 & 21 & 120 & 4\\
6 & 360 & 55 & 180 & 7
\end{tabular}
\caption{Degrees and monodromy timings for differential signatures.}\label{fig:diff-monodromy}
\end{figure}
\begin{figure}
\begin{tabular}{c|c|c|c|c|c|c}
$d$ & $\deg \JE$ & time (s) & $\deg_{\bs{e_1}} \JE $ & time (s) & $\deg_{\bs{e_2}} \JE $ & time (s) \\
\hline
%$\deg \mathcal{S}$  & time (s) & $\deg_{\bs{e_3}} \mathcal{S}$ & time (s) \\
2 & 42       & 4 & 24 & 2 & 26 & 2\\
3 & 936      & 33 & 576 & 17 & 696 & 16\\
4 & 3024     & 139 & 1920 & 57 & 2448 & 87 \\
5 & 7440     & 463 & 4800 & 206 & 6320 & 276\\
6 & 15480    & 1315 & 10080 & 748 & 13560 & 791
\end{tabular}
\caption{Degrees and monodromy timings for joint signatures (see Conjecture~\ref{conjecture:degrees}.)}\label{fig:joint-monodromy}
\end{figure}
Figures~\ref{fig:diff-monodromy} and~\ref{fig:joint-monodromy} give degrees and single-run timings for monodromy computations on curves up to degree $6$ under the Euclidean differential and joint signatures. We also considered multiprojective witness sets for $\SigE\subset \CC^1 \times \CC^1$ and $\JE\subset (\CC^1)^6,$ where fewer witness points are needed. For the differential signatures, we considered $(1,0)$-slices which fix the value of the squared curvature $K_1.$ For Euclidean joint signatures, there are two combinatorially distinct classes of $(\CC^1)^6$ witness sets determined by which $d_{i,j}$ are fixed; the undirected graph of fixed distances must either be the $3$-pan (a $3$-cycle with pendant edge) or the $4$-cycle. We fix corresponding multidimensions $\bs{e}_1=(1,1,1,1,0,0)$ and $\bs{e}_2=(0,1,1,1,1,0).$

The timings in figures~\ref{fig:diff-monodromy} and~\ref{fig:joint-monodromy} are not optimal for a number of reasons. For instance, some multiprojective witness sets have an \emph{imprimitive} monodromy action, meaning that additional symmetries can be exploited ~\cite{AR16}. We successfully ran monodromy (with less conservative settings) for both signature maps on curves of degree up to $10.$ These computations suggested formulas for the degrees. For the Euclidean joint signature, we state these formulas in the form of a conjecture. For the case of Euclidean differential signatures, see~\cite{KRV18}; degrees for $d=2$ are corrected by a factor of $4.$% (counting the isometries of a generic plane conic.)
\begin{conj}
  \label{conjecture:degrees}
Let $\JE_d$ denote the Euclidean joint signature for a generic plane curve of degree $d.$ For $d\ge 3$:
\begin{itemize}
\item[] $\deg \overline{\JE_d}= 12 d (d^3-1)$
\item[] $\deg_{\bs{e}_1} \JE_d = 8 d^2 (d^2-1)$ %for $\bs{e}=(1,1,1,1,0,0)$ (``cherry removed.'')
\item[] $\deg_{\bs{e}_2} \JE_d = 4 d (d-1) (3d^2+d-1).$ %for $\bs{e}=(0,1,1,1,1,0)$ (``disjoint edges removed.'')
\end{itemize}
\end{conj}
To assess the speed and robustness of the online equality test, we conducted an experiment where, for degrees $d = 2,\ldots , 6$, curves $C_1,\ldots , C_{10}$  were generated with coefficients drawn uniformly from the unit sphere in $\mathbb{R}^{(d+2)(d+1)/2}.$ For each $C_i,$ we computed a witness set via parameter homotopy from a generic degree $d$ curve. We then applied $20$ random transformations from $\E (\RR)$ to the $C_i$ and perturbed the resulting coefficients by random real $\vec{\epsilon}$ with $\lVert\vec{\epsilon}\rVert_2 \in \{ 0, 10^{-7}, 10^{-6}, \ldots, 10^{-3}\},$ thus obtaining curves $\widetilde{C_{i,1,\epsilon}}, \ldots , \widetilde{C_{i,20,\epsilon}}.$ With all numerical tolerances fixed, we ran the equality test for each $\widetilde{C_{i,j,\epsilon}}$ against each $C_i.$ 

\begin{figure}
\begin{tabular}{c|c|c|c|c}
$d$ & track time (ms) & lookup time (ms) & track $K_1$ & lookup $K_1$\\
2 & 191 & 0.35 & 127 & 0.25\\
3 & 177 & 0.37 & 121 & 0.31\\
4 & 276 & 0.42 & 145 & 0.36\\
5 & 472 & 0.39 & 203 & 0.43\\
6 & 597 & 0.40 & 284 & 0.37
\end{tabular}
\caption{Equality test timings  for Euclidean differential signatures $\SigE.$}\label{fig:equality-diff}
\end{figure}

\begin{figure}
\begin{tabular}{c|c|c|c|c}
$d$ & track time (ms) & lookup time (ms) & track $\bs{e}_1$ & lookup $\bs{e}_1$\\
2 & 230 & 0.36 & 208 & 0.34\\
3 & 283 & 0.38 & 213 & 0.35\\
4 & 335 & 0.39 & 288 & 0.40\\
5 & 409 & 0.32 & 357 & 0.32\\
6 & 507 & 0.32 & 462 & 0.33
\end{tabular}
\caption{Equality test timings for Euclidean joint signatures $\JE.$}\label{fig:equality-joint}
\end{figure}

Figures~\ref{fig:equality-diff} and~\ref{fig:equality-joint} summarize the timings for the equality tests in this experiment. Overall, these tests run on the order of sub-seconds. Most of the time is spent on path-tracking. The tracking times reported give the total time spent on lines 1 and 5 of Algorithm~\ref{alg:equality-test}. The only other possible bottleneck is the lookup on line 7. This is negligible, even for large witness set sizes, if an appropriate data structure is used. The runtimes for all cases considered seem comparable, although using differential signatures and multiprojective slices appear to give a slight edge over the respective alternatives.

\begin{figure}[H]
\begin{center}
  \begin{tabular}{cc}
\includegraphics[width=0.4\textwidth, height=0.4\textwidth]{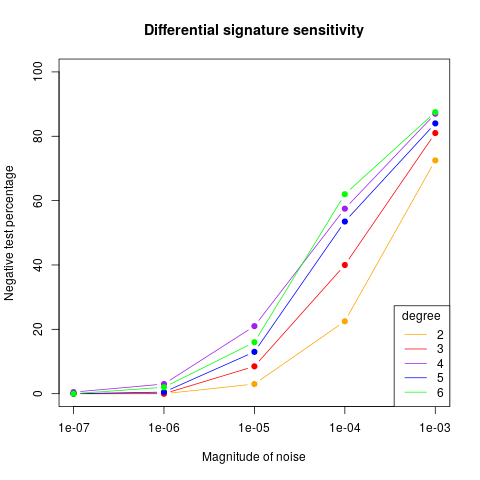} &
\includegraphics[width=0.4\textwidth, height=0.4\textwidth]{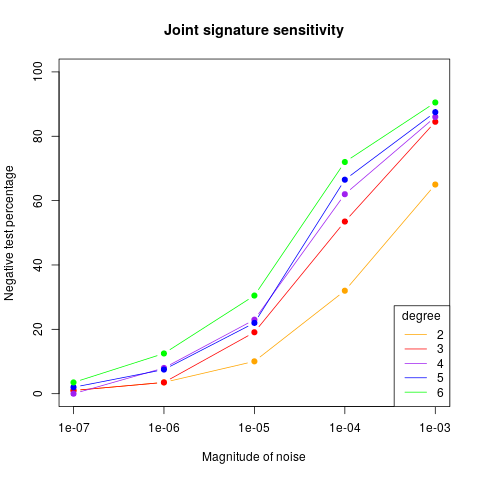}
    \\
\includegraphics[width=0.4\textwidth, height=0.4\textwidth]{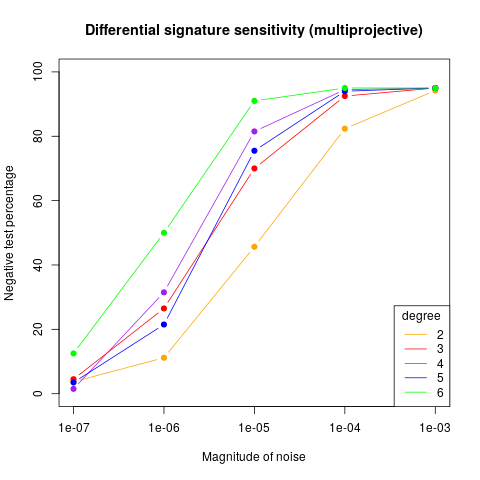} &
\includegraphics[width=0.4\textwidth, height=0.4\textwidth]{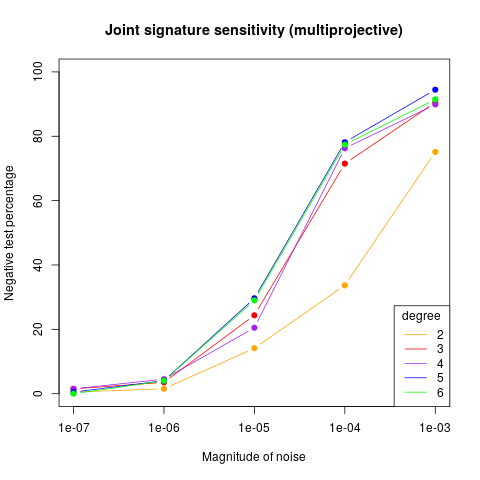}
  %% old
%\includegraphics[width=0.22\textwidth, height=0.23\textwidth]{../PIX/diff-standard-sensitivity.png} &
%\includegraphics[width=0.22\textwidth, height=0.23\textwidth]{../PIX/joint-standard-sensitivity.png}
    \\
  \end{tabular}
\end{center}
\caption{Sensitivity of the equality test on Euclidean signatures to noise.}\label{fig:sens-plot}
\end{figure}

The plots in Figure~\ref{fig:sens-plot} illustrate the results of our sensitivity analysis. The respective axes are the magnitude of the noise $\epsilon $ and the percentage of $C_{i,j,\epsilon}$ deemed to be not equivalent to $C_i.$ Note that the horizontal axis is given on a log scale, and excludes the noiseless case $\epsilon =0$; for this case, among all tests in the experiment, only one false negative was reported for the differential signatures with $d=6.$ We include a trend line to make the plots more readable. In general, we observe a threshold phenomenon, where most tests are positive for sufficiently low noise and are negative for sufficiently high noise. Besides the multiprojective differential signature (depicted in the bottom-left), we observe a similar stability profile for this type of random perturbation.

\begin{remark}
The thresholds in these experiments clearly depend on the numerical tolerances used (for this experiment, defaults are provided by \verb|NAG4M2|), the type of map, and the type of witness set.
\end{remark}

In Figure~\ref{fig:sens-EqAff-plot}, we reproduce the previous experiment for curves of degrees $d=3,4,5$ under $\SA(\CC)$. Perhaps unsurprisingly due to the higher degree of the image and the complexity of evaluating the signature maps, the equality test in this case is much more sensitive to small perturbations.
Here we observe a significant difference in the sensitivity between the equi-affine joint and differential equality tests.
In contrast to the Euclidean case, the joint signature appears to be \emph{far less} sensitive.
We also now observe in around 2\% of cases overall that there are \emph{failures} due to path-tracking, resulting in neither an equivalent nor inequivalent outcome.
We again exclude the noiseless case $\epsilon =0$ in these graphs where the false negative rate was  less than 1\%. 
Surprisingly, we also observed a non-negligible rate of ``false-positives'' for the $\SA (\CC)$ joint signature, wherein some $C_i$ and $C_j$ are declared equivalent.
We also note that we do not have an analogue of Conjecture~\ref{conjecture:degrees} for $\JA ,$ leaving us less certain about the completeness of the witness sets collected.

\begin{figure}[H]
\begin{center}
  \begin{tabular}{cc}
\includegraphics[width=0.4\textwidth, height=0.4\textwidth]{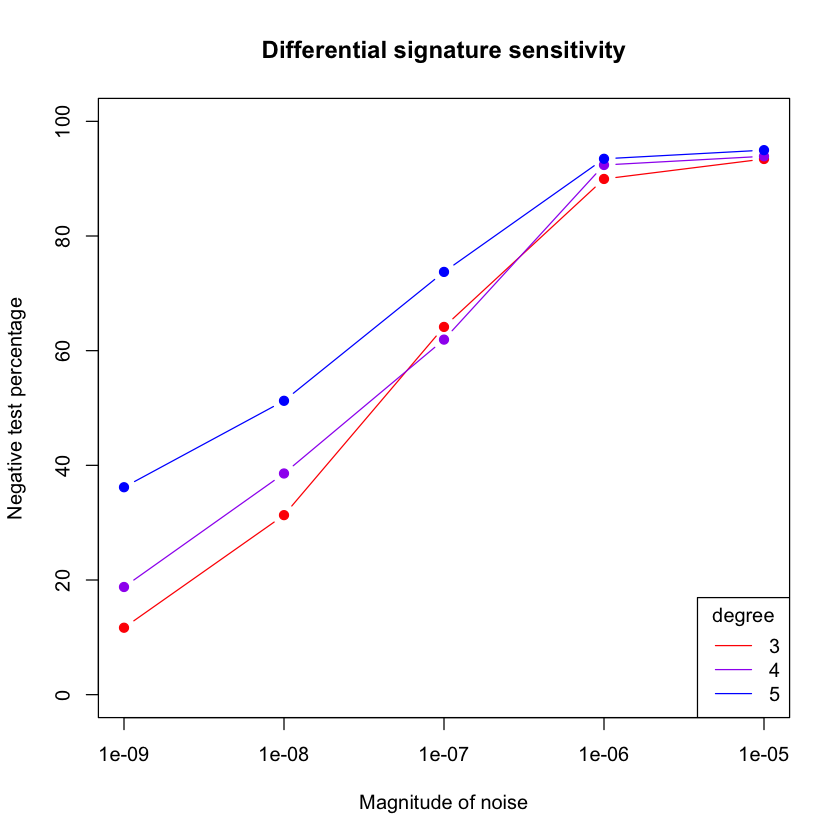} &
\includegraphics[width=0.4\textwidth, height=0.4\textwidth]{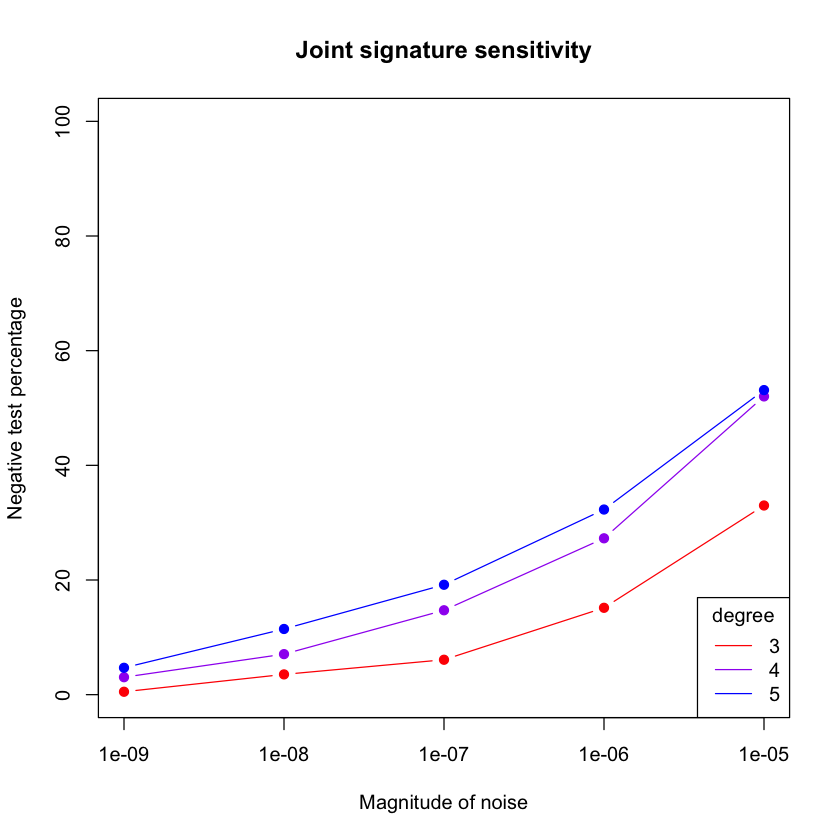}
    \\
  \end{tabular}
\end{center}
\caption{Sensitivity of the equality test on equi-affine signatures to noise.}\label{fig:sens-EqAff-plot}
\end{figure}

Finally we conduct the same experiment for the Euclidean differential and joint signatures under a different scheme of noise, with a view towards applications like curve-matching~\cite{HO13}.
Instead of perturbing the coefficients of the algebraic curve, we sample $\begin{pmatrix}d+2\\2\end{pmatrix}+1$ points on curves $C_1,\hdots, C_{10}$, perturb these points by $\vec{\epsilon} \in \RR^2$ with $|\vec{\epsilon }|=\epsilon ,$ and then reconstruct a new algebraic curve of the same degree through \emph{interpolation} before applying a random transformation from $\E(\RR)$.
Specifically, the equation defining our interpolated curve comes from singular vectors of the \emph{Vandermonde matrix} of all degree-$\le d$ monomials evaluated at the samples, as in~\cite{LVS}.
We emphasize that the coefficients of the perturbed curves have a more complicated dependence on $\epsilon $ in this experiment.
Moreover, we caution that our results may also depend on the number of points sampled from each curve.
Still, we find that the observations from this new experiment, with a more meaningful model of noise, and our original experiment are roughly consistent.

\begin{figure}[H]
\begin{center}
  \begin{tabular}{cc}
\includegraphics[width=0.4\textwidth, height=0.4\textwidth]{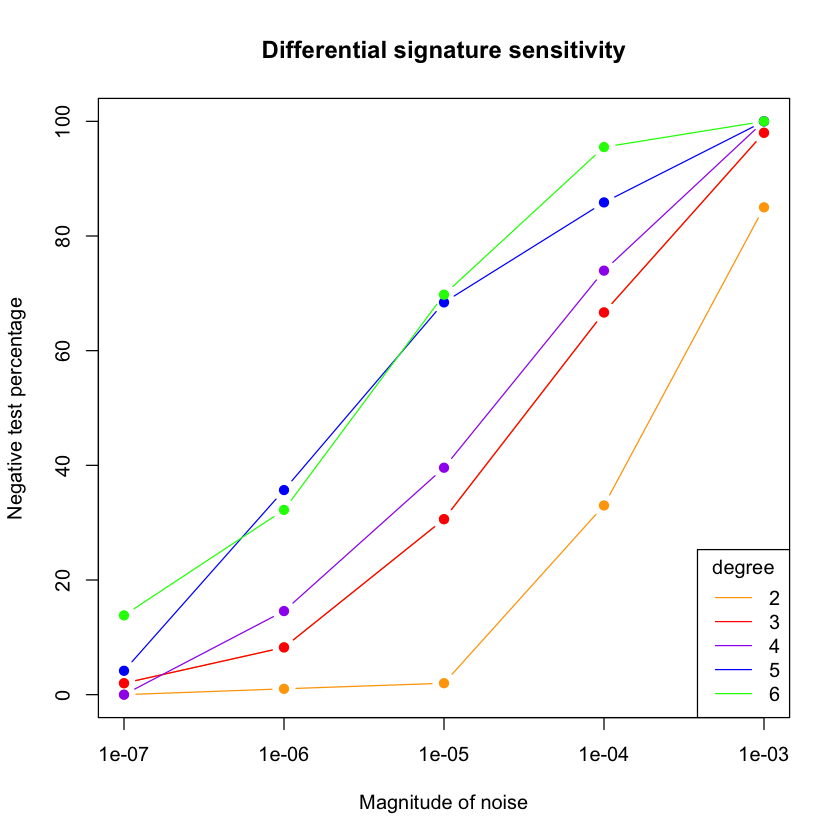} &
\includegraphics[width=0.4\textwidth, height=0.4\textwidth]{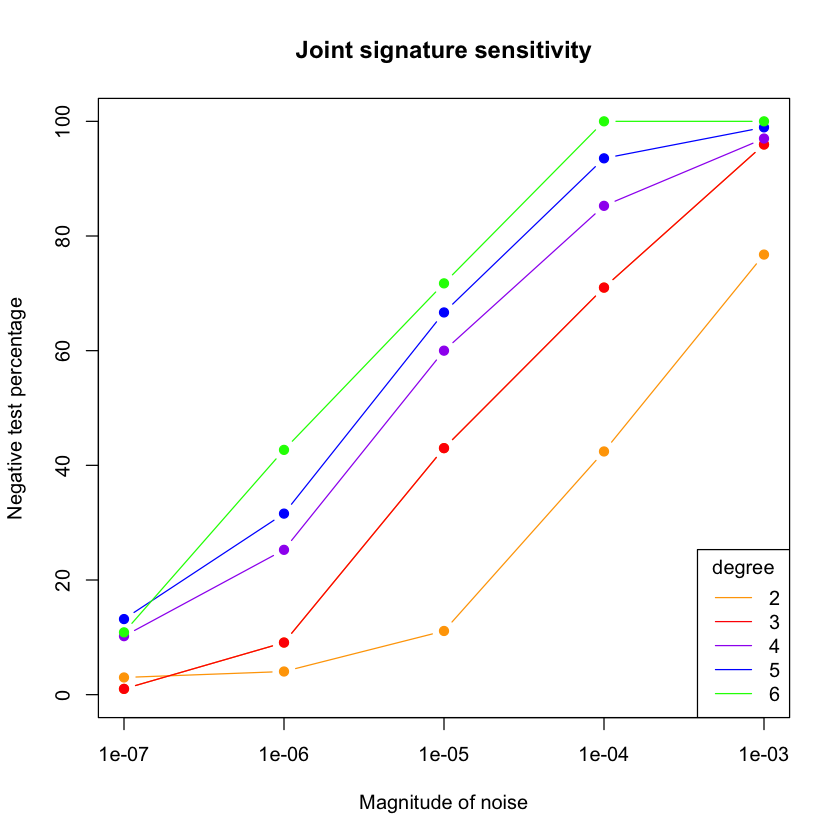}
    \\
  \end{tabular}
\end{center}
\caption{Sensitivity of the equality test for Euclidean signatures of curves computed from noisy samples.}\label{fig:sens-geometric-plot}
\end{figure}

In closing, we have shown that numerical algebraic geometry gives an effective way of solving the group equivalence problem for plane algebraic curves.
Our results open up new avenues of mathematical research, indicated at the end of Section~\ref{subsec:witness} and in Conjecture~\ref{conjecture:degrees}.
We also considered the effects of noise which might be relevant in applications.
In general, our methods seem to be brittle against significant levels of noise.
Nonetheless, we hope our efforts motivate work on the applications of curve signatures in the future.

%%
%% The acknowledgments section is defined using the "acks" environment
%% (and NOT an unnumbered section). This ensures the proper
%% identification of the section in the article metadata, and the
%% consistent spelling of the heading.
\section*{Acknowledgments}
{\small Research of T.~Duff is supported in part by NSF DMS-1719968, a fellowship from the Algorithms and Randomness Center at Georgia Tech, and by the Max Planck Institute for Mathematics in the Sciences in Leipzig.}
{\small Research of M.~Ruddy was supported in part by the Max Planck Institute for Mathematics in the Sciences in Leipzig.}

%%
%% The next two lines define the bibliography style to be used, and
%% the bibliography file.
%\bibliographystyle{ACM-Reference-Format}
\bibliographystyle{acm}
\bibliography{art}

\end{document}